\documentclass[11pt,a4paper]{smfart} 
\usepackage{mathsmf}
\usepackage{smfthm}
\usepackage{graphicx}
\usepackage{color}
\usepackage[utf8]{inputenc}
\usepackage[T1]{fontenc}
\usepackage[french]{babel}
\usepackage[plainpages]{hyperref}
\title[Spectre de Steklov et nombre chromatique]{Multiplicité du 
spectre de Steklov sur les surfaces et nombre chromatique} 
\author{Pierre Jammes}
\address{Univ. Nice Sophia Antipolis, CNRS, LJAD, UMR 7351\\
06100 Nice France}
\email{pjammes@unice.fr}
\begin{document}
\begin{abstract}
Dans cet article, on démontre plusieurs
résultats sur la multiplicité des premières valeurs propres de Steklov
sur les surfaces compactes à bord. On améliore certaines bornes sur
la multiplicité, 
en particulier pour la première valeur propre, et on montre qu'elles sont
optimales sur  plusieurs surfaces de petit genre. Dans un
article précédent, on a défini un nouvel invariant chromatique des surfaces
à bord et on a conjecturé qu'il est relié à la multiplicité de 
la première valeur propre de Steklov. Dans le present article on étudie
cet invarianti et on confirme la conjecture dans les cas où la borne 
optimale est connue.
\end{abstract}

\begin{altabstract}
In this article, we prove several results about the 
multiplicity of the first Steklov eigenvalues on compact surfaces with 
boundary. We improve some bounds on the multiplicity, especially for
the first eigenvalue, and we prove they are sharp on some surfaces
of small genus. In a previous article, we defined a new chromatic invariant
of surfaces with boundary and conjectured that this invariant is related
to the bound on the first eigenvalue. In the present article, we study
this invariant, and prove that the conjecture is true when the known bound 
is sharp.
\end{altabstract}
\altkeywords{Steklov eigenvalues, multiplicity, chromatic number}


\maketitle

\section{Introduction}
L'étude de la multiplicité des valeurs propres de Steklov a fait 
récemment l'objet des travaux \cite{fs12}, \cite{ja14}, 
\cite{kkp12}. Ces trois articles montrent que sur une surface compacte
à bord donnée, la multiplicité de la $k$-ième valeur propre de Steklov 
est majorée en fonction de $k$ et de la topologie (voir le paragraphe
qui suit pour la définition du spectre de Steklov). Dans \cite{ja14}, 
j'ai montré que ce phénomène est spécifique à la dimension~2 et qu'en
dimension plus grande, on peut prescrire arbitrairement le début
du spectre de Steklov, avec multiplicité. J'ai aussi conjecturé qu'en
dimension~2, la multiplicité maximale de la première valeur propre non nulle
est déterminée par un invariant topologique de même nature que le nombre 
chromatique mais spécifique aux surfaces à bord (cf. la 
conjecture~\ref{intro:conj} ci dessous).

L'objet de cet article est triple. D'abord, améliorer certaines bornes
sur la multiplicité du spectre de Steklov en dimension~2, en particulier pour
la première valeur propre et sur les surfaces de petit genre. Ensuite,
construire des exemples de première valeur propre multiple, ce qui permet
de montrer que certaines bornes sont optimales. Enfin, on étudie 
l'invariant chromatique des surfaces à bord définie dans \cite{ja14}, 
qu'on appellera \emph{nombre chromatique relatif}, on justifiera 
en particulier cette dénomination et on calculera sa valeur sur les
surfaces de petit genre.

 La confrontation de ces différents résultats permet de consolider 
la conjecture faite dans \cite{ja14} sur le lien entre
la multiplicité 
de la première valeur propre et le nombre chromatique relatif en montrant 
qu'elle est vérifiée dans 
tous les cas où on connaît la borne optimale et qui sont résumés
dans la table~\ref{intro:table} ci-dessous.

\subsection{Définitions et notations}
Avant d'énoncer les résultats précis, nous allons rappeler les notions
en jeu et préciser quelques définitions.

Soit $M$ une variété compacte à bord et $\gamma\in C^\infty(M)$, 
$\rho\in C^\infty(\partial M)$ deux fonctions densités strictement positives 
sur $M$ et $\partial M$ respectivement (on peut travailler avec des 
hypothèses de régularité plus faible sur les densités mais ça n'est pas
crucial dans la suite).
Le problème aux valeurs propres de Steklov consiste à résoudre l'équation,
d'inconnues $\sigma\in\R$ et $f:\overline M\to\R$,
\begin{equation}
\left\{\begin{array}{ll}
\divergence(\gamma\nabla f)=0 & \textrm{dans }M\\
\gamma\frac{\partial f}{\partial \nu}=\sigma \rho f & \textrm{sur }\partial M
\end{array}\right.
\end{equation}
où $\nu$ est un vecteur unitaire sortant normal au bord. On parle de 
problème de Steklov homogène quand $\gamma\equiv1$ et $\rho\equiv1$, et le
cas $\gamma\not\equiv1$ se rattache au problème de Calder\'on. 
L'ensemble des
réels $\sigma$ solutions du problème forme un spectre discret positif
noté
\begin{equation}
0=\sigma_0(M,g,\rho,\gamma)<\sigma_1(M,g,\rho,\gamma)\leq
\sigma_2(M,g,\rho,\gamma)\ldots
\end{equation}
On omettra les références à $\rho$ et $\gamma$ quand ces densités
sont uniformément égales à~1.

On s'intéressera à la multiplicité de ces valeurs propres, et si on se
donne une surface compacte à bord $\Sigma$, on notera $m_k(\Sigma)$ la
multiplicité maximale de $\sigma_k(\Sigma,g,\rho,\gamma)$ quand on 
fait varier $g$, $\rho$ et $\gamma$.

On établira des liens entre ce spectre et les graphes plongés dans la surface.
On introduit pour cela les deux définitions suivantes :

\begin{definition}
Soit $\Gamma$ un graphe fini et $\Sigma$ une surface compacte à bord.
Un plongement de $\Gamma$ dans $\Sigma$ sera appelé \emph{plongement propre}
si ce plongement envoie tous les sommets de $\Gamma$ sur $\partial\Sigma$.
\end{definition}
\begin{rema}
Les graphes qui se plongent proprement dans le disque sont connus
sous le nom de graphes planaires extérieurs (\emph{outerplanar graphs}).
\end{rema}

\begin{definition}
Si $\Sigma$ est une surface compacte à bord, on appelle \emph{nombre
chromatique relatif} de $\Sigma$, noté $\mathrm{Chr}_0(\Sigma)$,
la borne supérieure des nombres chromatiques des graphes finis qui admettent
un plongement propre dans $\Sigma$. 
\end{definition}

Avec ces notations, la conjecture énoncée dans \cite{ja14} peut se
reformuler ainsi (l'équivalence entre les deux énoncés est l'un
des objets de la section~\ref{relatif}, cf. remarque~\ref{rel:rem1}):
\begin{conj}\label{intro:conj}
Pour toute surface $\Sigma$ compacte à bord, on a $$m_1(\Sigma)=
\mathrm{Chr}_0(\Sigma)-1.$$
\end{conj}
Cette conjecture adapte au spectre de Steklov une conjecture analogue 
énoncée par Y.~Colin de~Verdière dans \cite{cdv87} et reliant la
multiplicité de la 2\ieme{} valeur propre des opérateurs de Schrödinger
et le nombre chromatique usuel.

Enfin, on utilisera les notations suivantes : $\mathbb S^2$, $\mathbb D^2$, 
$\mathbb P^2$, $\mathbb M^2$, $\mathbb T^2$ et $\mathbb K^2$ désigneront 
respectivement la sphère de dimension~2, le disque, le plan projectif, 
le ruban de Möbius, le tore et la bouteille de Klein. Si $\Sigma$ est une 
surface close, on désignera par $\Sigma_p$ la surface $\Sigma$ privée de 
$p$ disques disjoints, où $p$ est un entier strictement positif. Par exemple, 
on a $\mathbb M^2=\mathbb P^2_1$ et $\mathbb D^2=\mathbb S^2_1$.

\subsection{Résultats}

Un premier résultat de cet article est de démontrer qu'on a au moins
une inégalité dans l'égalité de la conjecture~\ref{intro:conj}:
\begin{theo}\label{intro:thpresc}
Pour toute surface $\Sigma$ une surface compacte à bord, on a
$$m_1(\Sigma)\geq\mathrm{Chr}_0(\Sigma)-1.$$
\end{theo}

On va s'attacher ensuite à améliorer les bornes connues sur la multiplicité
des valeurs propres, en commençant par considérer une surface quelconque. 
Rappelons d'abord le résultat de \cite{kkp12}, qui est la meilleure
majoration connue pour une surface $\Sigma$ et une valeur propre
$\sigma_k$ quelconques (elle est démontré dans \cite{kkp12} pour
$\gamma\equiv1$ mais sans hypothèse de régularité sur $\rho$):
\begin{theo}[\cite{kkp12}]\label{intro:kkp}
Si $\Sigma$ est une surface compacte à bord de caractéristique d'Euler
$\chi$ et dont le bord possède $l$ composantes connexes, alors
\begin{equation}\label{intro:kkp1}
m_k(\Sigma)\leq 2k-2\chi-2l+5
\end{equation}
et
\begin{equation}\label{intro:kkp2}
m_k(\Sigma)\leq k-2\chi+4
\end{equation}
pour tout $k\geq1$, cette dernière égalité étant stricte pour les
surfaces non simplement connexes.

\end{theo}
Les majorations de multiplicité que nous allons montrer sont de plusieurs 
types. Dans les deux théorèmes qui suivent, nous donnons deux majorations 
de $m_k$ 
sur n'importe quelle surface: le premier précise l'inégalité 
(\ref{intro:kkp2}) en utilisant les outils développés dans \cite{fs12}, 
le deuxième théorème est une application des techniques
développées par B.~Sévennec (\cite{se94}, \cite{se02}) pour majorer 
la multiplicité de la première valeur propre du laplacien.
Nous donnerons ensuite des majorations spécifiques aux premières valeurs 
propres des surfaces de petit genre.
\begin{theo}\label{intro:bornes}
Si $\Sigma$ est une surface compacte à bord de caractéristique d'Euler
$\chi$ et dont le bord possède $l$ composantes connexes, alors 
pour tout $k\geq1$ :
\begin{equation}\label{intro:borne1}
m_k(\Sigma)\leq k-2\chi+3.
\end{equation}
\end{theo}
\begin{rema}
La majoration~(\ref{intro:borne1}) est optimale pour $k=1$ sur le disque
et le ruban de Möbius, et dans les deux cas l'égalité est atteinte par un
problème homogène, c'est-à-dire avec $\rho\equiv1$ et $\gamma\equiv1$ 
(pour la métrique canonique dans le cas du disque et
pour une métrique $S^1$-invariante construite dans~\cite{fs12} pour 
le ruban de Möbius). 
\end{rema}
\begin{rema}
L'inégalité (\ref{intro:borne1}) étend les cas d'inégalité stricte 
de (\ref{intro:kkp2}) à toutes les surfaces. Les techniques de \cite{fs12} 
qui autorisent cette démonstration unifiée permettent aussi les raffinements
obtenus dans les théorèmes~\ref{intro:thp} et~\ref{intro:thdisque}.
\end{rema}
Contrairement au théorème~\ref{intro:bornes} et aux travaux~\cite{fs12},
\cite{kkp12} et~\cite{ja14}, le théorème qui suit n'utilise pas
les résultats de Cheng \cite{ch76} sur la structure de l'ensemble nodal
des fonctions propres mais les techniques topologiques de B.~Sévennec, ce qui
autorise des hypothèses de régularité beaucoup plus faibles (voir
le théorème~\ref{bornes:sev} et la remarque~\ref{bornes:reg}).
\begin{theo}\label{intro:sev}
Soit $\Sigma$ une surface compacte à bord de caractéristique d'Euler
$\chi$ et dont le bord possède $l$ composantes connexes.
Si $\chi+l\leq-1$, alors
\begin{equation}\label{intro:borne2}
m_1(\Sigma)\leq 5-\chi-l.
\end{equation}
Si $l=1$ et $\chi\leq-2$, alors
\begin{equation}\label{intro:borne3}
m_1(\Sigma)\leq 3-\chi.
\end{equation}
Ces majorations restent valables si $g$, $\rho$ et $\gamma$ sont $L^\infty$
sans régularité supplémentaire.
\end{theo}
\begin{rema}
Si on note $\bar\chi$ la caractéristique d'Euler de la surface close 
obtenue en collant un disque sur chaque bord de $\Sigma$, la 
majoration~(\ref{intro:borne2}) s'écrit $m_1(\Sigma)\leq 5-\bar\chi$. 
Elle est optimale pour $\bar\chi=-1$ et $l\geq2$, et $\bar\chi=-2$ 
ou~$-3$ et $l\geq3$. L'inégalité~(\ref{intro:borne3}) peut s'écrire
$m_1(\Sigma)\leq4-\bar\chi$, elle est optimale pour $\bar\chi=1,2,3$.
Les valeurs propres multiples correspondantes sont données par le 
théorème~\ref{intro:thpresc} et le calcul des nombres chromatiques
relatifs fait dans la section~\ref{relatif}.
\end{rema}

Dans les trois théorèmes qui suivent, on donne des bornes spécifiques
aux surfaces de petit genre. Le premier adapte des résultats analogues
obtenus par G.~Besson \cite{be80} et Y.~Colin de Verdière \cite{cdv87} 
pour le laplacien.
\begin{theo}\label{intro:thl}
$m_1(\mathbb T^2_p)=6$ pour $p\geq3$ et 
$m_1(\mathbb K^2_p)=5$ pour $p\geq2$.
\end{theo}
Avec des techniques similaires mais en utilisant des arguments spécifiques 
au spectre de Steklov, on améliore deux bornes:
\begin{theo}\label{intro:thp}
 $m_1(\mathbb P^2_2)=4$ et $m_1(\mathbb T^2_1)=5$. 
\end{theo}
Enfin, on majore la multiplicité de la 2\ieme{} valeur propre du disque au
moyen d'une stratégie indépendante de toutes celles évoquées jusqu'ici.
\begin{theo}\label{intro:thdisque}
$m_2(\mathbb D^2)=2$.
\end{theo}
\begin{rema}
On sait déjà que $m_1(\mathbb D^2)=2$. Par ailleurs, M.~Karpukhin, G.~Kokarev
et I.~Polterovich montrent dans \cite{kkp12} que $m_k(\mathbb D^2)=2$ si 
$k$ est assez grand. Ces résultats laissent penser qu'on a 
$m_k(\mathbb D^2)=2$ quel
que soit $k$. Cette multiplicité est atteinte par la
métrique canonique.
\end{rema}

La table~\ref{intro:table} rassemble les valeurs connues de $m_1(\Sigma_p)$,
où $\Sigma$ est une surface close (dans cette table, la notation $\#n\Sigma$ 
désigne la somme connexe de $n$ copies de la surface $\Sigma$). Pour les 
surfaces $\Sigma$ concernées, 
la valeur de $m_1$ ne dépend pas de $p$ quand p$\geq4$.
Rappelons que la majoration de $\mathbb S^2_p$ pour $p\geq2$ est montrée
dans \cite{fs12}, \cite{ja14} et \cite{kkp12} et l'égalité découle du
théorème~\ref{intro:thpresc} pour $p$ quelconque et d'un exemple 
construit dans \cite{fs11} pour le cas $p=2$.
\begin{table}[h]
\begin{center}
\renewcommand{\arraystretch}{1.5}
\begin{tabular}{c|c|c|c|c|c|}
$p=$ & 1 & 2 & 3 & 4 \\\hline\hline
$\mathbb S^2$ & 2 & 3 & 3 & 3 \\\hline
$\mathbb P^2$ & 4 & 4 & 5 & 5 \\\hline
$\mathbb K^2$ & ? & 5 & 5 & 5 \\\hline
$\mathbb T^2$ & 5 & ? & 6 & 6 \\\hline
$\#3\mathbb P^2$ & 5 & 6 & 6 & 6 \\\hline
$\mathbb T^2\#\mathbb T^2$ & 6 & 7 & 7 & 7 \\\hline
$\#4\mathbb P^2$ & 6 & ? & 7 & 7 \\\hline
$\#5\mathbb P^2$ & 7 & ? & 8 & 8 \\\hline
\end{tabular}
\medskip
\renewcommand{\arraystretch}{1}
\caption{Valeur de $m_1$ sur les surfaces de petit genre}
\label{intro:table}
\end{center}
\end{table}
Toutes les bornes contenues dans la table~\ref{intro:table} sont bien
conformes à la conjecture~\ref{intro:conj} (comparer avec les nombres
chromatiques relatifs indiqués dans la table~\ref{rel:table},
section~\ref{relatif}).

Les majorations de multiplicité pour la première valeur propre du laplacien
induisent des critères de plongement de graphes dans les surfaces. Dans
le cas du spectre de Steklov, on peut en tirer des critères de plongement
propre dans les surfaces à bord. Cet aspect sera précisé dans la 
section~\ref{presc} (corollaire~\ref{presc:crit}). Comme par ailleurs
on sait caractériser spectralement les graphes qui admettent un plongement
non entrelacé dans $\R^3$, on obtient, par un cheminement inattendu, le critère
de plongement qui suit. Rappelons qu'un plongement d'un graphe dans $\R^3$
est non entrelacé si tout ensemble de cycles disjoints forme un entrelacs
trivial.
\begin{cor}\label{intro:plongement}
Si un graphe admet un plongement propre dans $\mathbb M^2$ ou $\mathbb P^2_2$,
alors il admet un plongement non entrelacé dans $\R^3$.
\end{cor}
\begin{rema}
L'hypothèse que les sommets du graphe soient sur le bord 
est indispensable ; sans elle on peut trouver facilement des contre-exemples,
à commencer par le graphe complet à 6 sommets.
\end{rema}
\begin{rema}
Si la conjecture~\ref{intro:conj} est vraie pour $\mathbb K^2_1$, alors
le corollaire~\ref{intro:plongement} s'applique aussi à cette surface. 
\end{rema}
On discutera dans la section~\ref{presc} (remarque~\ref{plongement:rem}) 
l'existence de démonstrations
plus élémentaires de ce corollaire.

 Dans la section~\ref{rappels}, nous rappellerons quelques résultats 
techniques concernant le spectre de Steklov et les opérateurs sur les graphes. 
La section~\ref{relatif}
sera consacrée à l'étude du nombre chromatique relatif : on justifiera 
en particulier cette dénomination et on calculera sa valeur sur les 
surfaces de petit genre. On verra dans la section~\ref{presc} comment
construire des valeurs propres multiples sur une surface à l'aide d'un
graphe proprement plongé et on en déduira le théorème~\ref{intro:thpresc}
et le corollaire~\ref{intro:plongement}. Enfin, on démontrera dans la 
section~\ref{bornes} les différentes bornes sur la multiplicité.

Je remercie Y.~Colin de Verdière et I.~Polterovich pour leurs commentaires
sur la première version de cet article, ainsi qu'un rapporteur anonyme
dont les remarques ont permis d'améliorer le texte.

\section{Rappels}
\subsection{Le spectre de Steklov}\label{rappels}
On va rappeler ici quelques propriétés du spectre de Steklov que nous 
utiliserons. Il est l'ensemble des réels $\sigma$ pour lesquels le
problème

\begin{equation}
\left\{\begin{array}{ll}
\divergence(\gamma\nabla f)=0 & \textrm{dans }M\\
\gamma\frac{\partial f}{\partial \nu}=\sigma \rho f & \textrm{sur }\partial M
\end{array}\right.,
\end{equation}
où $\nu$ est un vecteur unitaire sortant normal au bord, admet
des solutions non triviales. Les variétés que nous considérerons auront
un bord $C^1$ par morceaux, ce qui est suffisant pour que le problème soit
bien défini.

Le spectre de Steklov $(\sigma_k(M,g,\gamma,\rho))_k$ est le spectre 
d'un opérateur Dirichlet-Neumann $H^1(\partial M)
\to L^2(\partial M)$ défini par $\Lambda_{\rho,\gamma} u=\frac\gamma\rho
\frac{\partial\mathcal H_\gamma u}{\partial\nu}$, où $\mathcal H_\gamma u$ est 
le prolongement harmonique de $u$ pour la densité $\gamma$, c'est-à-dire que 
$\divergence(\gamma\nabla(\mathcal H_\gamma u))=0$. Il est auto-adjoint 
pour la norme de Hilbert $\|u\|^2=\int_{\partial M}u^2\rho\,\de v_g$ (voir
\cite{ba80}, \cite{su90} et \cite{uh09} ; ces références traitent les
cas $\rho\equiv 1$ ou $\gamma\equiv 1$ mais l'adaptation au cas général 
est aisée).

Pour montrer le théorème~\ref{intro:thpresc} on utilisera la caractérisation
variationnelle suivante du spectre :

\begin{equation}\label{rappels:minmax}
\sigma_k(M,g,\rho)=\inf_{V_{k+1}\in H^1(M)}
\sup_{f\in V_{k+1}\backslash\{0\}}
\frac{\int_M|\de f|^2\gamma\ \de v_g}%
{\int_{\partial M}f^2\rho\ \de v_g},
\end{equation}
où $V_k$ parcours les sous-espaces de dimension~$k$ de l'espace 
de Sobolev $H^1(M)$. 

 On aura aussi recours à un problème de Steklov avec condition de Neumann
sur une partie du bord. Si on partitionne $\partial M$ en deux domaines
(ou unions finies de domaines) $\partial_S M$ et $\partial_N M$, ce problème
consiste à considérer la variante suivante du problème de Steklov:
\begin{equation}\label{rappels:eqsn}
\renewcommand{\arraystretch}{1.5}
\left\{\begin{array}{ll}
\divergence(\gamma\nabla f)=0& \textrm{dans }M\\
\gamma\frac{\partial f}{\partial \nu}=\sigma \rho f & \textrm{sur }
\partial M_S\\
\frac{\partial f}{\partial \nu}=0 & \textrm{sur }\partial M_N
\end{array}\right.
\renewcommand{\arraystretch}{1}
\end{equation}
c'est-à-dire qu'on demande à la fonction harmonique $f$ de vérifier la
condition de Neumann sur $\partial M_N$. Cette condition revient à poser
$\rho\equiv0$ sur $\partial M_N$. Le spectre obtenu est celui
d'un opérateur Dirichlet-Neumann défini sur $\partial_S M$. On notera
$(\sigma_k(M,\partial M_S,g,\gamma,\rho))_k$ son spectre.

\subsection{Opérateurs sur les graphes}\label{rappels:graphes}
On rappelle ici la définition des laplaciens et des opérateurs de Schrödinger
sur les graphes, en se référant par exemple à \cite{cdv98}.

Soit $\Gamma$ un graphe fini et $S$ son ensemble de sommets.
La géométrie du graphe est déterminée par la donnée,
pour chaque arête $a$ reliant deux sommets $x$ et $y$, d'un réel $l_a$ qui
représente la longueur de cette arête. L'ensemble des $l_a$
sera appelé une métrique sur le graphe $\Gamma$. Le laplacien agissant
sur les fonctions $f:S\to S$ est alors défini par
\begin{equation}
\Delta f (x)=\sum_{a\sim x}\frac{f(x)-f(y_a)}{l_a},
\end{equation}
où la somme porte sur l'ensemble des arêtes d'extrémité $x$ et où $y_a$ 
désigne l'autre extrémité de $a$. Cet opérateur
a un spectre positif, sa plus petite valeur propre est simple et vaut~0.

Si on se donne une fonction $V$ sur l'ensemble des sommets, on peut
définir un opérateur de Schrödinger $H_V$ sur $\Gamma$ par 
$H_Vf(x)=\Delta f(x)+V(x)f(x)$. Une propriété de ces opérateurs nous 
sera utile : si la première valeur propre de $H_V$ est nulle,
alors son spectre est celui de la forme quadratique d'un laplacien
relativement à une norme de Hilbert $|f|=\sum_{x\in S} \mu_x f^2(x)$.
Les coefficients $\mu_x$ s'interprètent comme une mesure sur l'espace
des sommets du graphe.

\section{Nombre chromatique relatif d'une surface à bord}\label{relatif}
On va montrer dans cette section différents résultats concernant le
nombre chromatique relatif des surfaces à bord, et en particulier calculer
cet invariant sur les surfaces de petit genre.

Rappelons d'abord que le nombre chromatique d'un surface close $\Sigma$, 
c'est-à-dire la borne supérieure des nombres chromatiques des graphes 
qu'on peut plonger dans $\Sigma$, est donné par la formule 
\begin{equation}
\mathrm{Chr}(\Sigma)=\lfloor\frac{7+\sqrt{49-24\chi(\Sigma)}}2\rfloor
\end{equation}
sauf pour la bouteille de Klein, pour laquelle $\mathrm{Chr}(\mathbb K^2)=6$.
L'étude de cet invariant, amorcée par P. J. Heawood \cite{he90} et L. Heffter 
\cite{he91} a été achevée par G.~Ringel et J.~Youngs \cite{ry68} en genre
non nul, et K.~Appel et W.~Haken \cite{ah76} pour la sphère. On peut consulter
\cite{ry68} pour un survol historique.

Notre premier résultat sur le nombre chromatique relatif est un encadrement
analogue à la majoration du nombre chromatique obtenue par Heawood dans
\cite{he90}. Rappelons que si $\Sigma$ est une surface close, on note
$\Sigma_p$ la surface obtenue en lui enlevant $p$ disques disjoints.
\begin{theo}\label{rel:th1}
Le nombre chromatique relatif $\mathrm{Chr}_0(\Sigma_p)$ possède les 
propriétés suivantes :
\begin{enumerate}
\item[(a)] $\mathrm{Chr}_0(\Sigma_p)$ est une
fonction croissante de $p$ et vérifie les inégalités
\begin{equation}
\mathrm{Chr}(\Sigma)-1\leq\mathrm{Chr}_0(\Sigma_p)\leq
\inf\left(\mathrm{Chr}(\Sigma),\frac{5+\sqrt{25-24\chi(\Sigma)+24p}}2\right);
\end{equation}
\item[(b)] $\mathrm{Chr}_0(\Sigma_p)$ est le nombre de sommets du plus
grand graphe complet proprement plongeable dans $\Sigma_p$;
\item[(c)] $\mathrm{Chr}_0(\Sigma_1)=\mathrm{Chr}(\Sigma)-1$ et
$\mathrm{Chr}_0(\Sigma_p)=\mathrm{Chr}(\Sigma)$ si
$p\geq(\mathrm{Chr}(\Sigma)-1)/2$.
\end{enumerate}
\end{theo}
\begin{rema}\label{rel:rem1}
Le fait que $\mathrm{Chr}_0(\Sigma_p)$ soit réalisé par un graphe complet
établit l'équivalence entre la conjecture~\ref{intro:conj} et la 
conjecture énoncée dans \cite{ja14}.
\end{rema}
Ce théorème permet de calculer la valeur exacte du nombre chromatique
relatif sur un certain nombre de surfaces. On va compléter cette liste
par des surfaces de petit genre. D'abord dans le cas où la caractéristique
d'Euler de $\Sigma$ vérifie $\chi(\Sigma)\geq-1$. On aura alors 
la valeur de $\mathrm{Chr}_0(\Sigma_p)$ pour tout $p$: 
\begin{prop}\label{rel:th2}
$\mathrm{Chr}_0(\mathbb P^2_2)=5$, $\mathrm{Chr}_0(\mathbb K^2_2)=6$ et
$\mathrm{Chr}_0(\#3\mathbb P^2_2)=7$.
\end{prop}
On va aussi calculer quelques nombres chromatiques relatifs supplémentaires
dans le cas où $-2\geq\chi(\Sigma)\geq-7$. La liste ne sera pas exhaustive
mais elle contiendra tous les cas pour lesquels on sait démontrer
la conjecture~\ref{intro:conj} :
\begin{prop}\label{rel:th3}
Soit $\Sigma$ une surface close. 
\begin{enumerate}
\item $\mathrm{Chr}_0(\#2\mathbb T^2_2)=\mathrm{Chr}_0(\#4\mathbb P^2_3)=8$.
\item Si $\chi(\Sigma)=-3$ ou $-4$, alors $\mathrm{Chr}_0(\Sigma_3)=9$.
\item Si $\chi(\Sigma)=-5$ alors $\mathrm{Chr}_0(\Sigma_4)=10$.
\item $\mathrm{Chr}_0(\#4\mathbb T^2_3)=10$.
\end{enumerate}
\end{prop}

La table~\ref{rel:table} rassemble en fonction de $\Sigma$ et $p$ 
les nombres chromatiques qu'on peut calculer à l'aide du 
théorème~\ref{rel:th1}, des propositions~\ref{rel:th2} et~\ref{rel:th3}
et en utilisant la monotonie du nombre chromatique relatif par rapport
à $p$ et par somme connexe. On se limite aux surfaces $\Sigma$ 
telles que $\chi(\Sigma)\geq-7$.
\begin{table}[h]
\begin{center}
\renewcommand{\arraystretch}{1.5}
\begin{tabular}{c|c|c|c|c|c|}
$p=$ & 1 & 2 & 3 & 4 & 5\\\hline\hline
$\mathbb S^2$ & 3 & 4 & 4 & 4 & 4\\\hline
$\mathbb P^2$ & 5 & 5 & 6 & 6 & 6\\\hline
$\mathbb K^2$ & 5 & 6 & 6 & 6 & 6\\\hline
$\mathbb T^2$ & 6 & 6 & 7 & 7 & 7\\\hline
$\chi=-1$ & 6 & 7 & 7 & 7 & 7\\\hline
$\#2\mathbb T^2$ & 7 & 8 & 8 & 8 & 8\\\hline
$\#4\mathbb P^2$ & 7 & ? & 8 & 8 & 8\\\hline
$\chi=-3$ & 8 & 8 & 9 & 9 & 9\\\hline
$\chi=-4$ & 8 & ? & 9 & 9 & 9\\\hline
$\chi=-5$ & 9 & 9 & 9 & 10 & 10\\\hline
$\#4\mathbb T^2$ & 9 & 9 & 10 & 10 & 10\\\hline
$\chi=-7$ & 9 & ? & 10 & 10 & 10\\\hline

\end{tabular}
\renewcommand{\arraystretch}{1}
\bigskip
\caption{Nombre chromatique relatif des surfaces de petit genre}
\label{rel:table}
\end{center}
\end{table}

On va montrer séparément le résultat le plus technique du
théorème~\ref{rel:th1}:
\begin{lemme}\label{rel:lem1}
$\displaystyle\mathrm{Chr}_0(\Sigma_p)\leq
\frac{5+\sqrt{25-24\chi(\Sigma)+24p}}2.$
\end{lemme}
\begin{proof}
On pose $c_p= \lfloor(5+\sqrt{25-24\chi(\Sigma)+24p})/2\rfloor$.
Étant donné un graphe $\Gamma$ plongé proprement dans $\Sigma_p$,
on va montrer qu'on peut le colorier avec $c_p$ couleurs.
Quitte à ajouter (temporairement) des sommets et des arêtes, 
on peut supposer d'une part
que toutes les faces de la décomposition de $\Sigma_p$ induite par
$\Gamma$ sont simplement connexes, et d'autre part que
$\partial\Sigma_p$ est recouvert par des arêtes. Il peut apparaître des 
arêtes multiples (plusieurs arêtes ayant les mêmes sommets) ou des boucles 
(arête reliant un sommet à lui-même) lors de cette étape mais ça n'est pas 
gênant pour la suite. On note $s$ et $a$ le
nombre de sommets et d'arêtes de $\Gamma$, $\delta$ le degré minimal
de ses sommets et $f$ le nombre de faces de la décomposition de
$\Sigma_p$ induite par $\Gamma$.

On commence par majorer $\delta$. Comme chaque face est bordée par au
moins trois arêtes et que chaque arête (sauf celles qui sont
sur le bord, qui sont en nombre égal au nombre de sommets) est adjacente 
à deux faces, on a $3f\leq2a-s$. De plus, le
degré minimal vérifie les inégalités $\delta s\leq2a$ et $\delta+1\leq s$.
Il découle alors de la formule d'Euler-Poincaré, en notant
$\chi=\chi(\Sigma)-p$ la caractéristique d'Euler de $\Sigma_p$ et
en supposant que $\delta\geq4$, que
\begin{eqnarray}
6\chi & = & 6f+6s-6a\leq 4s-2a\nonumber\\
& \leq & 4s-\delta s=(4-\delta)s\leq(4-\delta)(\delta+1)\nonumber\\
& \leq & -\delta^2+3\delta+4.
\end{eqnarray}
Comme $\delta$ vérifie l'inéquation $\delta^2-3\delta-4+6\chi\leq0$,
on en déduit que $\delta\leq \lfloor(3+\sqrt{25-24\chi})/2\rfloor=c_p-1$. Si
$\delta\leq4$, cette inégalité est trivialement vérifiée dès que
$\chi\leq0$.

On supprime ici les boucles qui sont apparues
en ajoutant des arêtes dans l'étape précédente ; la majoration de $\delta$
reste valide.

 On construit ensuite un coloriage par une double récurrence sur $p$, et sur 
$s$ à $p$ fixé. Dans la récurrence sur $s$, on aura besoin que la majoration
de $\delta$ s'applique bien à tous les graphes considérés.

Si $s\leq c_p$, le graphe $\Gamma$ se colorie évidemment avec $c_p$ couleurs.
Dans le cas contraire, on note $x$ un sommet 
de degré minimal et on considère le graphe $\Gamma'$ obtenu en supprimant 
le sommet $x$ et les arêtes qui lui sont adjacentes. La récurrence se
décompose en deux cas, selon que la composante de bord où est situé $x$
porte d'autres sommets de $\Gamma$ où non. Soulignons que si $p=1$, 
tous les sommets sont sur la même composante de bord donc seul le premier
cas se présentera.

 Dans le premier cas, on ajoute une arête qui joint les deux voisins de $x$
sur la même composante de bord (s'il ne reste qu'un seul point sur cette
composante, l'arête supplémentaire forme une boucle) et on appelle
encore $\Gamma'$ le graphe obtenu. La démonstration de la majoration 
$\delta\leq c_p-1$ obtenue pour $\Gamma$ s'applique alors aussi à $\Gamma'$. 
Comme $\Gamma'$ a $s-1$ 
sommets, on peut lui appliquer l'hypothèse de récurrence et le colorier 
avec $c_p$ couleurs. On applique le même coloriage
à tous les sommets de $\Gamma$ sauf $x$, et comme le degré $\delta$ de $x$ est
majoré par $c_p-1$ on peut colorier $x$ avec une couleur différente de ses
voisins.

 Dans le second cas, la composante de bord de $x$ ne porte aucun sommet
de $\Gamma'$ et on se ramène à la surface $\Sigma_{p-1}$ en collant un disque
le long de cette composante de bord. Le graphe $\Gamma'$ est 
proprement plongé dans $\Sigma_{p-1}$, donc admet un coloriage à $c_{p-1}$ 
couleurs par hypothèse de récurrence. Comme la constante $c_p$ croît avec
$p$ il admet un coloriage
à $c_p$ couleurs. On conclut comme dans le cas précédent.
\end{proof}

\begin{proof}[Démonstration du théorème~\ref{rel:th1}]
On va noter temporairement $\kappa(\Sigma_p)$ le nombre de sommets
du plus grand graphe complet proprement plongeable dans $\Sigma_p$. Il 
est clair qu'un plongement dans $\Sigma_p$
induit un plongement dans $\Sigma_{p+1}$ et dans $\Sigma$, donc
$\kappa(\Sigma_p)$ et $\mathrm{Chr}_0(\Sigma_p)$ sont des fonctions
croissantes de $p$ et 
\begin{equation}\label{rel:ineg1}
\kappa(\Sigma_p)\leq\mathrm{Chr}_0(\Sigma_p)\leq\mathrm{Chr}(\Sigma).
\end{equation}

Supposons que $p=1$. Si on se donne un plongement propre du graphe complet 
$K_n$ à $n$ sommets dans $\Sigma_1$, on peut
construire un plongement de $K_{n+1}$ dans $\Sigma$ en collant
un disque sur le bord de $\Sigma_1$, en ajoutant un sommet sur ce disque
et en le reliant par des arêtes aux sommets de $K_n$ (c'est possible
puisqu'ils sont tous sur le bord du disque). Réciproquement, étant donné
un graphe complet $K_{n+1}$ étant plongé dans $\Sigma$, on peut enlever
un disque contenant un seul sommet et déplacer les autres sommets de
manière à ce qu'ils se situent sur le bord du disque. On obtient
un plongement de $K_n$ dans $\Sigma_1$ dont les sommets sont sur 
$\partial\Sigma_1$. On en déduit que $\kappa(\Sigma_1)=\mathrm{Chr}(\Sigma)-1$
et donc que
\begin{equation}\label{rel:ineg2}
\mathrm{Chr}(\Sigma)-1\leq\kappa(\Sigma_p)\leq\mathrm{Chr}_0(\Sigma_p)\leq
\mathrm{Chr}(\Sigma),
\end{equation}
ce qui, avec le lemme~\ref{rel:lem1}, conclut la démonstration du point~(a) 
du théorème.

On peut maintenant montrer le point (b), c'est-à-dire que 
$\kappa(\Sigma_p)=\mathrm{Chr}_0(\Sigma_p)$.
Ces deux nombres ne peuvent prendre que les valeurs $\mathrm{Chr}(\Sigma)-1$ 
et $\mathrm{Chr}(\Sigma)$. Si $\mathrm{Chr}_0(\Sigma_p)=
\mathrm{Chr}(\Sigma)-1$, alors on a immédiatement l'égalité dans les
deux premières inégalités de~(\ref{rel:ineg2}). Supposons donc que 
$\mathrm{Chr}_0(\Sigma_p)=
\mathrm{Chr}(\Sigma)$ et considérons un graphe~$\Gamma$ proprement plongé 
dans $\Sigma_p$, dont le nombre
chromatique est $\mathrm{Chr}(\Sigma)$ et qui est critique, c'est-à-dire
qu'on ne peut pas lui enlever d'arêtes sans diminuer son nombre
chromatique. Cette propriété de criticité est intrinsèque au 
graphe~$\Gamma$, donc il est aussi critique comme graphe plongé dans~$\Sigma$.
Or, on sait qu'un tel graphe est nécessairement un graphe complet
(cf. \cite{bo98}, ch.~V, p.~156-157), donc on a égalité dans 
l'inégalité~(\ref{rel:ineg1}).

On a déjà montré que $\mathrm{Chr}_0(\Sigma_1)=\kappa(\Sigma_1)=
\mathrm{Chr}(\Sigma)-1$. Il reste à montrer que 
$\mathrm{Chr}_0(\Sigma_p)=\mathrm{Chr}(\Sigma)$ si
$p\geq(\mathrm{Chr}(\Sigma)-1)/2$ pour conclure le point (c) du théorème. 
On considère le graphe complet à 
$\mathrm{Chr}(\Sigma)$ sommets qu'on plonge dans $\Sigma$ et on 
construit un plongement de ce graphe dans $\Sigma_p$ en enlevant $p$
disques (pour $p$ suffisamment grand) à $\Sigma$ de manière à ce que les
disques ne rencontrent pas les arêtes et que les sommets soient situés
sur le bord des disques. Pour déterminer une valeur de $p$ adéquate,
on remarque qu'en plaçant le premier disque de manière quelconque, on peut
placer (au moins) trois sommets du graphe sur son bord. On peut ensuite
regrouper les autres sommets par deux et placer un disque adjacent 
à ces deux sommets et qui longe l'arête qui les relie. En fonction de 
la parité du nombre de sommets, il peut en rester un qui nécessite un disque
supplémentaire. Au total, on a utilisé $\lceil(\mathrm{Chr}(\Sigma)-1)/2
\rceil$ disques.
\end{proof}

\begin{proof}[Démonstration de la proposition~\ref{rel:th2}]
Le fait que $\mathrm{Chr}_0(\mathbb P^2_2)=5$ découle des
théorèmes~\ref{intro:thpresc} et~\ref{intro:thp}.

Le calcul de $\mathrm{Chr}_0(\mathbb K^2_2)$ repose sur une amélioration
de l'estimation faite par le théorème~\ref{rel:th1} dans le cas
où $p\geq\mathrm{(Chr(\Sigma)-1)/2}$ en utilisant le plongement explicite
d'un graphe complet dans la surface. Dans le cas de la bouteille de Klein,
le plongement du graphe complet à~6 sommets est représenté sur la 
figure~\ref{rel:fig0} (les lettres indiquent comment sont identifiés les
cotés droite et gauche).

\begin{figure}[h]
\begin{center}
\begin{picture}(0,0)%
\includegraphics{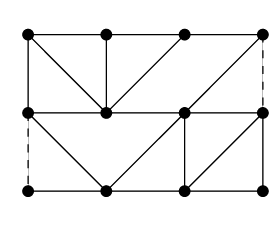}%
\end{picture}%
\setlength{\unitlength}{4144sp}%
\begingroup\makeatletter\ifx\SetFigFont\undefined%
\gdef\SetFigFont#1#2#3#4#5{%
  \reset@font\fontsize{#1}{#2pt}%
  \fontfamily{#3}\fontseries{#4}\fontshape{#5}%
  \selectfont}%
\fi\endgroup%
\begin{picture}(2254,1722)(436,-1030)
\put(1206,545){\makebox(0,0)[lb]{\smash{{\SetFigFont{11}{13.2}{\rmdefault}{\mddefault}{\updefault}{\color[rgb]{0,0,0}$2$}%
}}}}
\put(1206,-370){\makebox(0,0)[lb]{\smash{{\SetFigFont{11}{13.2}{\rmdefault}{\mddefault}{\updefault}{\color[rgb]{0,0,0}$5$}%
}}}}
\put(2399,545){\makebox(0,0)[lb]{\smash{{\SetFigFont{11}{13.2}{\rmdefault}{\mddefault}{\updefault}{\color[rgb]{0,0,0}$1$}%
}}}}
\put(1803,-52){\makebox(0,0)[lb]{\smash{{\SetFigFont{11}{13.2}{\rmdefault}{\mddefault}{\updefault}{\color[rgb]{0,0,0}$6$}%
}}}}
\put(610,545){\makebox(0,0)[lb]{\smash{{\SetFigFont{11}{13.2}{\rmdefault}{\mddefault}{\updefault}{\color[rgb]{0,0,0}$1$}%
}}}}
\put(1803,545){\makebox(0,0)[lb]{\smash{{\SetFigFont{11}{13.2}{\rmdefault}{\mddefault}{\updefault}{\color[rgb]{0,0,0}$3$}%
}}}}
\put(2519,-211){\makebox(0,0)[lb]{\smash{{\SetFigFont{11}{13.2}{\rmdefault}{\mddefault}{\updefault}{\color[rgb]{0,0,0}$4$}%
}}}}
\put(610,-966){\makebox(0,0)[lb]{\smash{{\SetFigFont{11}{13.2}{\rmdefault}{\mddefault}{\updefault}{\color[rgb]{0,0,0}$1$}%
}}}}
\put(1206,-966){\makebox(0,0)[lb]{\smash{{\SetFigFont{11}{13.2}{\rmdefault}{\mddefault}{\updefault}{\color[rgb]{0,0,0}$2$}%
}}}}
\put(1803,-966){\makebox(0,0)[lb]{\smash{{\SetFigFont{11}{13.2}{\rmdefault}{\mddefault}{\updefault}{\color[rgb]{0,0,0}$3$}%
}}}}
\put(2399,-966){\makebox(0,0)[lb]{\smash{{\SetFigFont{11}{13.2}{\rmdefault}{\mddefault}{\updefault}{\color[rgb]{0,0,0}$1$}%
}}}}
\put(451,-211){\makebox(0,0)[lb]{\smash{{\SetFigFont{11}{13.2}{\rmdefault}{\mddefault}{\updefault}{\color[rgb]{0,0,0}$4$}%
}}}}
\put(451,119){\makebox(0,0)[lb]{\smash{{\SetFigFont{9}{10.8}{\rmdefault}{\mddefault}{\updefault}{\color[rgb]{0,0,0}(a)}%
}}}}
\put(451,-511){\makebox(0,0)[lb]{\smash{{\SetFigFont{9}{10.8}{\rmdefault}{\mddefault}{\updefault}{\color[rgb]{0,0,0}(b)}%
}}}}
\put(2476,119){\makebox(0,0)[lb]{\smash{{\SetFigFont{9}{10.8}{\rmdefault}{\mddefault}{\updefault}{\color[rgb]{0,0,0}(b)}%
}}}}
\put(2476,-511){\makebox(0,0)[lb]{\smash{{\SetFigFont{9}{10.8}{\rmdefault}{\mddefault}{\updefault}{\color[rgb]{0,0,0}(a)}%
}}}}
\end{picture}%
\end{center}
\caption{Plongement de $K_6$ dans $\mathbb K^2$%
\label{rel:fig0}}
\end{figure}
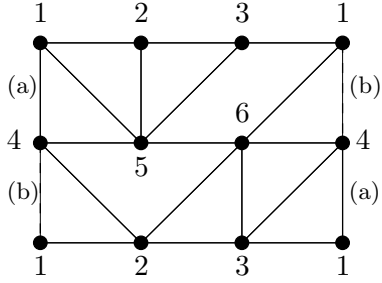

En enlevant les faces $(125)$ et $(346)$, on obtient un plongement
propre de $K_6$ dans $\mathbb K^2_2$.

Il reste à calculer $\mathrm{Chr}_0(\#3\mathbb P^2_2)$.
On sait déjà que $\mathrm{Chr}_0(\#3\mathbb P^2_2)\leq7$ 
(théorème~\ref{rel:th1}), il suffit donc d'exhiber un plongement propre de 
$K_7$.
On procède de la manière suivante : topologiquement, on enlève deux disques
au ruban de Möbius $\mathbb M^2$ et on colle un pantalon (une sphère privée 
de trois disques) sur les deux composantes de bord ainsi crées. Pour construire 
le graphe, on part du graphe $K_5$ plongé dans le ruban de Möbius et on
ajoute deux sommets situés sur le troisième bord du pantalon. Les détails
de la construction sont représentés sur la figure~\ref{rel:fig1}.
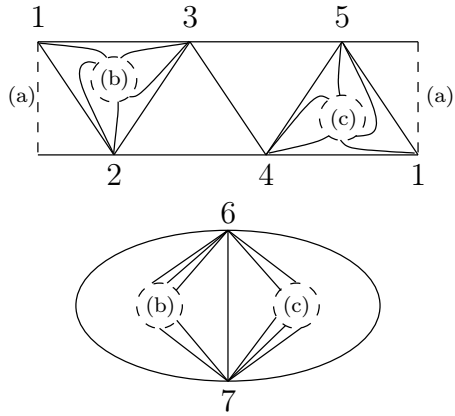
\begin{figure}[h]
\begin{center}
\begin{picture}(0,0)%
\includegraphics{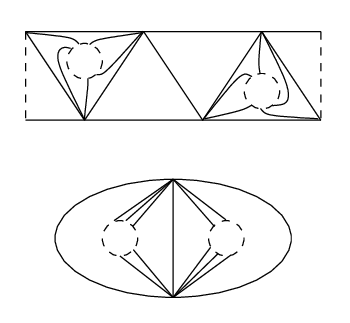}%
\end{picture}%
\setlength{\unitlength}{4144sp}%
\begingroup\makeatletter\ifx\SetFigFont\undefined%
\gdef\SetFigFont#1#2#3#4#5{%
  \reset@font\fontsize{#1}{#2pt}%
  \fontfamily{#3}\fontseries{#4}\fontshape{#5}%
  \selectfont}%
\fi\endgroup%
\begin{picture}(2658,2460)(256,-1831)
\put(406,479){\makebox(0,0)[lb]{\smash{{\SetFigFont{12}{14.4}{\rmdefault}{\mddefault}{\updefault}{\color[rgb]{0,0,0}1}%
}}}}
\put(1306,479){\makebox(0,0)[lb]{\smash{{\SetFigFont{12}{14.4}{\rmdefault}{\mddefault}{\updefault}{\color[rgb]{0,0,0}3}%
}}}}
\put(1756,-466){\makebox(0,0)[lb]{\smash{{\SetFigFont{12}{14.4}{\rmdefault}{\mddefault}{\updefault}{\color[rgb]{0,0,0}4}%
}}}}
\put(2206,479){\makebox(0,0)[lb]{\smash{{\SetFigFont{12}{14.4}{\rmdefault}{\mddefault}{\updefault}{\color[rgb]{0,0,0}5}%
}}}}
\put(2656,-466){\makebox(0,0)[lb]{\smash{{\SetFigFont{12}{14.4}{\rmdefault}{\mddefault}{\updefault}{\color[rgb]{0,0,0}1}%
}}}}
\put(1531,-691){\makebox(0,0)[lb]{\smash{{\SetFigFont{12}{14.4}{\rmdefault}{\mddefault}{\updefault}{\color[rgb]{0,0,0}6}%
}}}}
\put(856,-466){\makebox(0,0)[lb]{\smash{{\SetFigFont{12}{14.4}{\rmdefault}{\mddefault}{\updefault}{\color[rgb]{0,0,0}2}%
}}}}
\put(1531,-1816){\makebox(0,0)[lb]{\smash{{\SetFigFont{12}{14.4}{\rmdefault}{\mddefault}{\updefault}{\color[rgb]{0,0,0}7}%
}}}}
\put(814,128){\makebox(0,0)[lb]{\smash{{\SetFigFont{8}{9.6}{\rmdefault}{\mddefault}{\updefault}{\color[rgb]{0,0,0}(b)}%
}}}}
\put(1087,-1222){\makebox(0,0)[lb]{\smash{{\SetFigFont{8}{9.6}{\rmdefault}{\mddefault}{\updefault}{\color[rgb]{0,0,0}(b)}%
}}}}
\put(271, 29){\makebox(0,0)[lb]{\smash{{\SetFigFont{8}{9.6}{\rmdefault}{\mddefault}{\updefault}{\color[rgb]{0,0,0}(a)}%
}}}}
\put(2179,-106){\makebox(0,0)[lb]{\smash{{\SetFigFont{8}{9.6}{\rmdefault}{\mddefault}{\updefault}{\color[rgb]{0,0,0}(c)}%
}}}}
\put(1900,-1222){\makebox(0,0)[lb]{\smash{{\SetFigFont{8}{9.6}{\rmdefault}{\mddefault}{\updefault}{\color[rgb]{0,0,0}(c)}%
}}}}
\put(2746, 29){\makebox(0,0)[lb]{\smash{{\SetFigFont{8}{9.6}{\rmdefault}{\mddefault}{\updefault}{\color[rgb]{0,0,0}(a)}%
}}}}
\end{picture}%
\end{center}
\caption{Plongement propre de $K_7$ dans $\#3\mathbb P^2_2$%
\label{rel:fig1}}
\end{figure}

En haut de la figure est représenté le graphe $K_5$ plongé dans 
$\mathbb M^2$. On enlève deux disques situés dans les faces $(123)$ et
$(451)$. En bas est représenté le pantalon, les deux bords intérieurs
étant recollés sur $\mathbb M^2$. Les sommets~$6$ et~$7$ sont bien
reliés entre eux et aux cinq premiers sommets.
\end{proof}

\begin{proof}[Démonstration de la proposition~\ref{rel:th3}]
 Pour calculer les nombres chromatiques traités ici, on procède comme
on l'a fait pour $\mathrm{Chr}_0(\mathbb K^2_2)$ : on considère le découpage
de $\Sigma$ induit par le plongement d'un graphe complet maximal et on
cherche comment enlever des disques de manière à placer tous les sommets
sur le bord.

On commence par le cas $\chi=-2$. Des plongements de $K_8$ dans
$\mathbb T^2\#\mathbb T^2$ et $\#4\mathbb P^2_2$ ont été décrits
respectivement par L.~Heffter \cite{he91} et I.~N.~Kagno \cite{ka35}.
La figure~\ref{rel:fig2} représente un plongement de $K_8$ dans 
$\mathbb T^2\#\mathbb T^2$: la surface $\mathbb T^2\#\mathbb T^2$ 
est obtenue en recollant les deux pantalons représentés sur la figure
le long de leurs bords, la position des sommets de $K_8$ indiquant 
l'orientation du recollement. On peut remarquer que l'arête $(34)$
peut être placée arbitrairement d'un coté ou de l'autre du bord des
pantalons portant les deux sommets. En fonction de ce choix, on obtient
une surface orientable ou non.

\begin{figure}[h]
\begin{picture}(0,0)%
\includegraphics{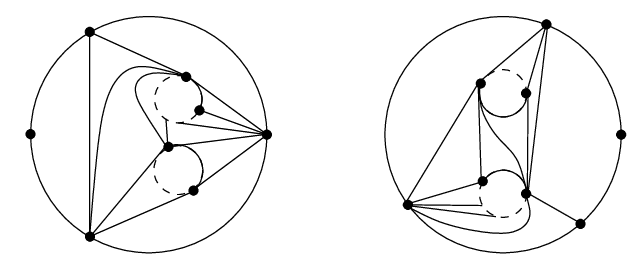}%
\end{picture}%
\setlength{\unitlength}{4144sp}%
\begingroup\makeatletter\ifx\SetFigFont\undefined%
\gdef\SetFigFont#1#2#3#4#5{%
  \reset@font\fontsize{#1}{#2pt}%
  \fontfamily{#3}\fontseries{#4}\fontshape{#5}%
  \selectfont}%
\fi\endgroup%
\begin{picture}(4823,2007)(443,-1504)
\put(987,341){\makebox(0,0)[lb]{\smash{{\SetFigFont{12}{14.4}{\rmdefault}{\mddefault}{\updefault}{\color[rgb]{0,0,0}$8$}%
}}}}
\put(901,-1440){\makebox(0,0)[lb]{\smash{{\SetFigFont{12}{14.4}{\rmdefault}{\mddefault}{\updefault}{\color[rgb]{0,0,0}$6$}%
}}}}
\put(2547,-575){\makebox(0,0)[lb]{\smash{{\SetFigFont{12}{14.4}{\rmdefault}{\mddefault}{\updefault}{\color[rgb]{0,0,0}$5$}%
}}}}
\put(458,-570){\makebox(0,0)[lb]{\smash{{\SetFigFont{12}{14.4}{\rmdefault}{\mddefault}{\updefault}{\color[rgb]{0,0,0}$7$}%
}}}}
\put(1834, -5){\makebox(0,0)[lb]{\smash{{\SetFigFont{8}{9.6}{\rmdefault}{\mddefault}{\updefault}{\color[rgb]{0,0,0}$1$}%
}}}}
\put(1813,-357){\makebox(0,0)[lb]{\smash{{\SetFigFont{8}{9.6}{\rmdefault}{\mddefault}{\updefault}{\color[rgb]{0,0,0}$2$}%
}}}}
\put(1984,-989){\makebox(0,0)[lb]{\smash{{\SetFigFont{8}{9.6}{\rmdefault}{\mddefault}{\updefault}{\color[rgb]{0,0,0}$4$}%
}}}}
\put(1531,-627){\makebox(0,0)[lb]{\smash{{\SetFigFont{8}{9.6}{\rmdefault}{\mddefault}{\updefault}{\color[rgb]{0,0,0}$3$}%
}}}}
\put(5251,-576){\makebox(0,0)[lb]{\smash{{\SetFigFont{12}{14.4}{\rmdefault}{\mddefault}{\updefault}{\color[rgb]{0,0,0}$5$}%
}}}}
\put(4617,356){\makebox(0,0)[lb]{\smash{{\SetFigFont{12}{14.4}{\rmdefault}{\mddefault}{\updefault}{\color[rgb]{0,0,0}$8$}%
}}}}
\put(4861,-1400){\makebox(0,0)[lb]{\smash{{\SetFigFont{12}{14.4}{\rmdefault}{\mddefault}{\updefault}{\color[rgb]{0,0,0}$6$}%
}}}}
\put(3343,-1260){\makebox(0,0)[lb]{\smash{{\SetFigFont{12}{14.4}{\rmdefault}{\mddefault}{\updefault}{\color[rgb]{0,0,0}$7$}%
}}}}
\put(4503,-962){\makebox(0,0)[lb]{\smash{{\SetFigFont{8}{9.6}{\rmdefault}{\mddefault}{\updefault}{\color[rgb]{0,0,0}$2$}%
}}}}
\put(3949,-149){\makebox(0,0)[lb]{\smash{{\SetFigFont{8}{9.6}{\rmdefault}{\mddefault}{\updefault}{\color[rgb]{0,0,0}$4$}%
}}}}
\put(4311,-242){\makebox(0,0)[lb]{\smash{{\SetFigFont{8}{9.6}{\rmdefault}{\mddefault}{\updefault}{\color[rgb]{0,0,0}$3$}%
}}}}
\put(3999,-855){\makebox(0,0)[lb]{\smash{{\SetFigFont{8}{9.6}{\rmdefault}{\mddefault}{\updefault}{\color[rgb]{0,0,0}$1$}%
}}}}
\end{picture}%
\begin{center}
\end{center}
\caption{Plongement de $K_8$ dans $\mathbb T^2\#\mathbb T^2$%
\label{rel:fig2}}
\end{figure}
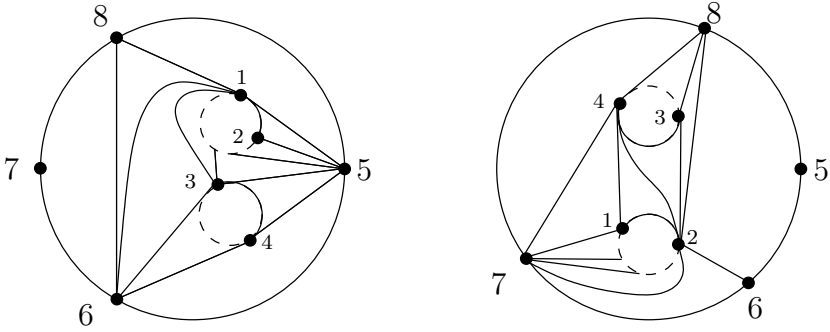

On peut noter que la décomposition de la surface induite par le plongement
comporte un quadrilatère $(5628)$. Si on enlève cette face, on peut ensuite
enlever deux autres disques, par exemple le long des arêtes $(13)$ et $(47)$,
de manière à obtenir un plongement propre de $K_8$ dans 
$(\mathbb T^2\#\mathbb T^2)_3$ ou $(\#4\mathbb P^2)_3$ selon le choix
d'orientation. Par conséquent, $\mathrm{Chr}_0((\mathbb T^2\#\mathbb T^2)_3)
=\mathrm{Chr}_0(\#4\mathbb P^2_3)=8$.

La figure~\ref{rel:fig2} ne permet pas de montrer que 
$\mathrm{Chr}_0((\mathbb T^2\#\mathbb T^2)_2)=8$ car les huit sommets ne 
sont pas sur le bord de deux quadrilatères. Cependant, le plongement de
$K_8$ dans $(\mathbb T^2\#\mathbb T^2)_2$ décrit dans \cite{ri74} (p.~23, 
table~(2.9)) vérifie cette propriété, les deux quadrilatères étant
$(0246)$ et $(1357)$.

Le plongement de $K_9$ dans $\#5\mathbb P^2$ a été construit par 
H.S.M.~Coxeter dans \cite{co43}. Ce plongement triangule la surface et 
on peut vérifier qu'on peut trouver trois faces dont les sommets 
sont tous distincts (par exemple les faces $(129)$, $(678)$ et $(345)$
de la figure~11 de \cite{co43}). En enlevant ces faces, on obtient
un plongement propre de $K_9$ dans $\#5\mathbb P^2_3$.

 Le plongement de $K_9$ dans $\#3\mathbb T^2$ exhibé par L.~Heffter 
dans \cite{he91} comporte une face hexagonale ayant 5~sommets distincts.
En enlevant cette face et deux autres disques, on peut donc obtenir
un plongement propre de $K_9$ dans $\#3\mathbb T^2_3$. Le cas de 
$\#6\mathbb P^2_3$ s'obtient par somme connexe de $\#5\mathbb P^2_3$
avec $\mathbb P^2$. 

Dans \cite{bo39}, R.~C. Bose a donné la construction d'un plongement
de $K_{10}$ dans $\#7\mathbb P^2_3$, elle est aussi étudiée (et illustrée
par une figure) par H.~Coxeter dans \cite{co43}. Le graphe triangule la 
surface et on peut trouver deux triangles n'ayant pas de sommets communs
(voir la figure~14 de \cite{co43}). Les quatre autres sommets du graphe
peuvent se placer sur le bord de deux disques ; en enlevant quatre disques,
on peut donc bien placer tous les sommets sur le bord.

 Le plongement de $K_{10}$ dans $\#4\mathbb T^2$ exhibé 
dans \cite{he91} comporte des quadrilatères, chacun ayant ses quatre sommets
distincts. Après avoir enlevé un de ces quadrilatères, il reste 6~sommets
à placer sur un bord, ce qu'on peut faire en enlevant deux autres faces 
triangulaires (les trois faces sont $(2,1,10,4)$, $(3,6,8)$, $(1,4,5)$ dans
la construction de Heffter).
\end{proof}

\section{Construction de valeurs propres multiples}\label{presc}
Cette section est consacrée à la démonstration du 
théorème~\ref{intro:thpresc}. Elle est très similaire à la démonstration
du théorème~1.2 de~\cite{ja14}. On va donc rappeler les différentes étapes
en indiquant les modifications à apporter. Le principe consiste à
faire tendre le début du spectre de la surface vers celui d'un
laplacien combinatoire d'un graphe complet, puis d'utiliser la
propriété de stabilité du spectre de ce graphe (cf.~\cite{cdv88} 
les rappels ci dessous) pour exhiber une valeur propre multiple sur
la surface.

Par cette méthode on va en fait montrer un résultat analogue pour un graphe
fini quelconque, cela nous sera utile en particulier pour démontrer
le corollaire~\ref{intro:plongement}. L'énoncé général fait intervenir
l'invariant de graphe forgé par Y.~Colin de Verdière dans 
\cite{cdv90} (voir aussi \cite{cdv98}). Rappelons que cet invariant,
qu'on notera $\mu(\Gamma)$, peut se définir comme étant la multiplicité
maximale de la deuxième valeur propre des opérateurs de Schrödinger
combinatoires sur le graphe $\Gamma$ pour lesquels cette valeur
propre multiple vérifie l'hypothèse de transversalité d'Arnol'd formalisée
dans~\cite{cdv88} (on dit alors que la multiplicité est stable). 
Si on note $M$ la matrice de cet opérateur de Schrödinger, cette 
hypothèse de tranvsersalité consiste en ce que dans l'espace des matrices
symétriques , le sous-espace des matrices d'opérateurs de Schrödinger 
et le sous-espace des matrices possédant une deuxième valeur propre de 
même multiplicité se coupent transversalement. En pratique, on
manipulera des laplaciens à densité sur le graphe plutôt que des opérateurs 
de Schödinger, les deux points de vue étant équivalents (cf. 
\textsection~\ref{rappels:graphes}). On peut consulter 
\cite{cdv88}, \cite{cdv98}
ou \cite{ja09b} pour une description plus détaillée de cette
propriété de transversalité et de son utilisation.

\begin{theo}\label{presc:th}
Soit $\Gamma$ un graphe admettant un plongement 
propre dans une surface compacte à bord $\Sigma$. Il existe des densités
$\gamma\in C^\infty(\Sigma)$, $\rho\in C^\infty(\partial\Sigma)$ et une 
métrique $g$ sur $\Sigma$  
telles que la multiplicité de $\sigma_1(\Sigma,g,\rho,\gamma)$ égale 
$\mu(\Gamma)$.
\end{theo}
Le théorème~\ref{intro:thpresc} s'en déduit en remarquant d'une part
qu'un graphe complet à $n$ sommets vérifie $\mu(K_n)=n-1$
(\cite{cdv88}, section~4, th.~1), et d'autre part que si 
$n=\textrm{Chr}_0(\Sigma)$, alors $K_n$ admet un plongement propre dans
$\Sigma$ (théorème~\ref{rel:th1} de la section précédente).

On peut aussi en déduire le critère de plongement qui suit :
\begin{cor}\label{presc:crit}
Si $\Gamma$ admet un plongement propre dans $\Sigma$, alors 
$\mu(\Gamma)\leq m_1(\Sigma)$.
\end{cor}
Le corollaire~\ref{intro:plongement} est une conséquence de ce dernier ;
il sera traité au paragraphe~\ref{presc:plongement}.

 La première étape de la démonstration du théorème~\ref{presc:th} consiste à 
construire une variété dont le début du spectre de Steklov tend vers celui 
d'un laplacien combinatoire pour une métrique donnée
sur $\Gamma$, cette variété pouvant s'interpréter comme un voisinage
tubulaire de $\Gamma$. En dimension~2, cette variété, qu'on 
notera~$\Omega$, est formée de $n$ demi-disques ($n$ étant le nombre
de sommets de $\Gamma$), deux demi-disques étant reliés par un rectangle
fin si les sommets correspondants sont reliés par une arête. Sur ce
domaine, on considère le spectre de Steklov-Neumann (voir les rappels
de la section~\ref{rappels}): le bord de Steklov est formé des diamètres
des demi-disques, et on pose la condition de Neumann sur le reste du bord.

\begin{lemme}[\cite{ja14}, théorème~4.3]\label{presc:lem1}
Il existe une famille de métriques $g_\varepsilon$ sur $\Omega$ telle que
les $n$ premières valeurs propre de Steklov de $(\Omega,g_\varepsilon)$ 
tendent vers le spectre de $\Delta_\Gamma$, avec convergence des espaces 
propres.
\end{lemme}
\begin{rema}
Dans \cite{ja14}, ce théorème est démontré pour des densités uniformes
sur $\partial\Omega_S$ et sur les sommets du graphes. Si on munit 
les sommets du graphe d'une autre mesure, on peut adapter la démonstration
en munissant chaque composante de $\partial\Omega_S$ (les diamètres des
demi-disques) de la densité correspondante. On peut donc faire tendre
le spectre de $\Omega$ non seulement vers le spectre d'un laplacien 
sur $\Gamma$, mais aussi d'un opérateur de Schrödinger.
\end{rema}

Le deuxième ingrédient de la démonstration est un résultat de 
convergence du spectre d'une surface vers celui d'un de ses domaines.
Sur un domaine $U$ d'une surface à bord $\Sigma$, on considérera 
le problème de Steklov-Neumann avec $\partial U_S=\partial U
\cap\partial\Sigma$ comme bord de Steklov et $\partial U_N=\partial U
\backslash \partial U_S$ comme bord de Neumann:
\begin{lemme}\label{presc:lemme}
Soit $(\Sigma,g)$ une surface riemannienne compacte à bord, $\rho$ une
densité sur $\partial \Sigma$ et $U$ un domaine 
de $\Sigma$ à bord $C^1$ par morceaux tel que $\partial U_S=\partial U
\cap\partial\Sigma$ soit non vide. Il existe des familles de densités
$\gamma_\varepsilon\in C^\infty(\Sigma)$ et de métriques $g_\varepsilon$ telles 
que $\sigma_k(\Sigma,g_\varepsilon,\gamma_\varepsilon,\rho)\to
\sigma_k(U,\partial U_S,g,\rho_{|\partial U_S})$ avec convergence 
des espaces propres quand $\varepsilon$ tend vers~0.
\end{lemme}
La démonstration est la même que celle du théorème~3.8 de \cite{ja14},
avec l'adaptation à la dimension~2 introduite par Y.~Colin de Verdière
dans \cite{cdv87}. On introduit une densité singulière
$\bar\gamma_\eta$ qui vaut~1 sur $U$ et $\eta$ sur~$\Sigma\backslash U$,
une métrique $\bar g_\eta$ égale à $g$ sur $U$ et à $\eta^2 g$ sur
$\Sigma\backslash U$
et on travaille avec les familles de formes quadratiques et de normes de 
Hilbert induites par $\bar\gamma_\eta$ et $\bar g_\eta$:
\begin{equation}
Q_\eta(f)=\inf_{\tilde f_{|\partial M}=f}
\left(\int_U|\de\tilde f|^2\de v_g+\eta^3
\int_{M\backslash U} |\de\tilde f|^2\de v_g\right)
\end{equation}
et $|f|_\eta=\int_{\partial U_S} f^2\de v_g+
\eta^2\int_{\partial M\backslash \partial U_S} f^2\de v_g$. Le reste
de la démonstration (convergence de spectre et lissage de la densité et
de la métrique)
est identique à~\cite{ja14}.

\begin{proof}[Démonstration du théorème~\ref{presc:th}]
Soit $\Gamma$ admettant un plongement propre dans $\Sigma$. On peut trouver 
un voisinage de $\Gamma$ dans $\Sigma$
difféomorphe au domaine $\Omega$ décrit précédemment. Étant donnée
une métrique et une mesure sur $\Gamma$, on peut trouver une famille de 
métriques $g_\varepsilon$ sur $\Omega$ dont le début du spectre de 
Steklov-Neumann tend vers le 
spectre du laplacien combinatoire sur $\Gamma$
(lemme~\ref{presc:lem1}). Pour toute métrique sur $\Omega$ on peut 
trouver une famille de densités sur $\Sigma$ et $\partial\Sigma$ telle 
que le spectre de Steklov de $\Sigma$ tende vers celui de $\Omega$ 
(lemme~\ref{presc:lemme}). Par conséquent,
il existe des familles de métriques $g_\varepsilon$ de densités 
$\rho_\varepsilon$ (étendant la métrique $g_\varepsilon$ et la densité
$\rho$ de $\Omega$) et de densités 
$\gamma_\varepsilon$ sur $\Sigma$ telles que le 
début du spectre de ($\Sigma, g_\varepsilon,\rho_\varepsilon,
\gamma_\varepsilon)$ tende vers le spectre 
du laplacien combinatoire sur $\Gamma$.

En munissant $\Gamma$ de la métrique et de la mesure
réalisant la multiplicité $\mu(\Gamma)$, on peut faire tendre le spectre de 
$\Sigma$ vers un spectre limite (la métrique et la densité étant alors
dégénérées) ayant la multiplicité souhaitée. Pour que 
$\sigma_1(\Sigma,g_\varepsilon,\rho_\varepsilon,\gamma_\varepsilon)$ soit 
de multiplicité $\mu(\Gamma)$ pour une 
métrique $g_\varepsilon$ et des densités $\rho_\varepsilon$ et 
$\gamma_\varepsilon$ lisses, on 
utilise la propriété de stabilité de la multiplicité: quitte à déformer 
l'opérateur sur $\Gamma$,
on peut trouver un $\varepsilon>0$ tel que $\sigma_1(\Sigma,g_\varepsilon,
\gamma_\varepsilon)$ soit de multiplicité~$\mu(\Gamma)$.
\end{proof}

\subsection{Critère de plongement non entrelacé}\label{presc:plongement}
 Cette section est consacrée au corollaire~\ref{intro:plongement}. Sa 
démonstration repose sur les propriétés
de l'invariant $\mu$ des graphes défini par Y.~Colin de Verdière et déjà 
utilisé dans le paragraphe précédent. On utilisera en particulier le fait que 
cet invariant permet de caractériser les graphes non entrelacés:
\begin{theo}[\cite{rst95},\cite{bcdv95},\cite{ls98}]\label{plongement:th}
Un graphe $\Gamma$ est non entrelacé si et seulement si $\mu(\Gamma)\leq4$.
\end{theo}
Le fait $\mu(\Gamma)\leq4$ implique le non entrelacement du graphe
découle de la caractérisation par mineurs exclus des graphes non 
entrelacés due à N.~Robertson, P.~Seymour et R.~Thomas \cite{rst95}
 et du calcul de $\mu$ sur les graphes de la famille
de Petersen fait dans \cite{bcdv95}. La réciproque a été montrée
par L.~Lov\'asz et A.~ Schrijver dans \cite{ls98}.

Le corollaire~\ref{intro:plongement} découle de la remarque suivante :
sachant que $m_1(\mathbb M^2)=m_1(\mathbb P_2^2)=4$, si $\Gamma$ admet 
un plongement propre dans $\mathbb M^2$ ou $\mathbb P^2_2$, alors 
$\mu(\Gamma)\leq4$ d'après le théorème~\ref{presc:th}. Le 
théorème~\ref{plongement:th} permet alors de conclure.

\begin{rema}\label{plongement:rem}
Si $\Gamma\to\Sigma$ est un plongement propre d'un graphe dans
une surface et $\Sigma\to\R^3$ un plongement de cette surface 
dans $\R^3$, on peut espérer par composition obtenir un plongement
non entrelacé de $\Gamma$. Si le plongement usuel $\mathbb M^2\to\R^3$
semble suffire pour tout plongement propre $\Gamma\to\mathbb M^2$,
il n'en va pas de même pour $\mathbb P_2^2$ : si on considère le plongement
$\mathbb P_2^2\to\R^3$ obtenu en enlevant un disque au plongement
$\mathbb M^2\to\R^3$, il existe des plongements propres 
$\Gamma\to\mathbb M^2$ tels que le plongement $\Gamma\to\R^3$ induit
soit entrelacé (on utilise le fait que le bord et l'âme du ruban de Möbius
forment un entrelacs non trivial). Mais cela n'exclut pas l'existence, 
pour un plongement $\Gamma\to\mathbb M^2$ fixé, d'un autre plongement 
$\mathbb P_2^2\to\R^3$ tel que le plongement de $\Gamma$ soit non 
entrelacé.
\end{rema}

\section{Bornes sur la multiplicité}\label{bornes}
\subsection{Propriétés de l'ensemble nodal d'une fonction propre}
Ce paragraphe sera consacré aux propriétés générales de l'ensemble
nodal des fonctions propres du problème de Steklov. Il s'agit généralement
de rappels de résultats montrés dans \cite{fs12}, \cite{ja14} et \cite{kkp12},
ou remontant à l'étude des fonctions propres du laplacien. Certaines sont
reformulées ou précisées. Elles sont regroupées en deux théorèmes : le
théorème~\ref{bornes:nodal1} rassemble les propriétés locales
de l'ensemble nodal et le théorème~\ref{bornes:nodal3} les propriétés
topologiques globales. Ces résultats ont généralement été démontrés
dans le cas homogène mais leur démonstration reste valide dans le 
cas général. Si une fonction $f$ s'annule en un point $p$, on appelle
ordre d'annulation de $f$ en $p$ le plus petit entier $k$ tel que
$\nabla^kf\neq0$.

\begin{theo}[\cite{ch76}, \cite{fs12}]\label{bornes:nodal1}
Soit $p$ un point de l'ensemble nodal d'une fonction propre $f$
du problème de Steklov sur $\Sigma$. On note $k$ l'ordre d'annulation 
de $f$ en $p$.

Au voisinage de $p$ l'ensemble nodal est la réunion de $k$ arcs de courbes 
s'intersectant en $p$, de courbure géodésique nulle en $p$ et
formant un système équiangulaire. De plus :
\begin{enumerate}
\item Si $p$ un point intérieur à $\Sigma$, alors $p$ est l'extrémité
de $2k$ arcs nodaux;
\item Si $p$ est sur le bord de $\Sigma$, alors $p$ est l'extrémité de
$k$ arcs nodaux rencontrant transversalement le bord.
\end{enumerate}
\end{theo}
On peut en déduire, comme dans~\cite{ch76}, que l'ensemble nodal est
la réunion de sous-variétés de dimension~1 immergées dans $\Sigma$,
chaque sous-variété étant soit un cercle ne rencontrant pas le bord,
soit un intervalle dont les extrémités sont sur le bord. On appellera
\emph{lignes nodales} ces sous-variétés. On peut aussi interpréter 
l'ensemble nodal comme un graphe plongé dans $\Sigma$, les sommets
étant les points critiques de la fonction propre. On appellera 
\emph{arêtes nodales} les arêtes de ce graphe. Soulignons que les arêtes
nodales sont plongées alors que les lignes nodales sont seulement immergées.

On aura aussi besoin du lemme qui suit et qui assure l'existence d'une
fonction propre s'annulant à un ordre élevé quand la multiplicité 
est grande :
\begin{lemme}\label{bornes:nodal2}
Soit $k$ un entier positif, $p$ un point de $\Sigma$ et $E$ un
espace  propre du problème de Steklov. Si l'une des deux conditions 
suivantes est vérifiée :
\begin{enumerate}
\item $p$ est un point intérieur à $\Sigma$ et $E$ est de dimension
au moins $2k$;
\item $p$ est sur le bord de $\Sigma$ et $E$ est de dimension au moins
$k+1$;
\end{enumerate}
alors $E$ contient une fonction qui s'annule à l'ordre $k$ en $p$.
\end{lemme}
\begin{proof}
Le cas où $p$ est intérieur à $\Sigma$ est traité dans \cite{be80}. 
On va donc supposer sur $p\in\partial\Sigma$ et que $E$ est de 
dimension au moins $k+1$. Par une déformation conforme, on peut identifier
localement la surface au demi-plan supérieur muni des coordonnées $(x,y)$ 
et le point $p$ au point $(0,0)$.

Comme $E$ est de dimension au moins $k+1$, on peut trouver une
fonction non nulle de $E$ telle que 
\begin{equation}
f(p)=\frac{\partial f}{\partial x}(p)=
\ldots=\frac{\partial^{k-1} f}{\partial x^{k-1}}(p)=0.
\end{equation}
 On sait aussi que 
$f$ vérifie l'équation aux valeurs propres $\partial f/\partial y=
\sigma\rho f$ en tout point du bord. En dérivant cette relation par
rapport à $x$ et en l'évaluant au point $p$, on obtient que 
\begin{equation}
\frac{\partial f}{\partial y}(p)=\frac{\partial^2 f}{\partial x\partial y}(p)
=\ldots=\frac{\partial^k f}{\partial x^{k-1}\partial y}(p)=0.
\end{equation}
Ces relations suffisent pour conclure si $k=1$ ou $2$.
 
Enfin, la fonction $f$ est harmonique relativement à la densité $\gamma$, 
c'est-à-dire qu'elle vérifie la 
relation $\divergence(\gamma\nabla f)=0$. Par conséquent,
$\frac{\partial^2 f}{\partial x^2}+\frac{\partial^2 f}{\partial y^2}$
peut s'écrire comme une expression d'ordre~0 ou~1 en $f$, et donc, 
en utilisant les relations déjà obtenues,
$\frac{\partial^2 f}{\partial y^2}(p)=0$ si $k\geq 3$.
En dérivant successivement l'équation $\divergence(\gamma\nabla f)=0$, on 
obtient par récurrence que toutes les dérivées partielles de $f$ d'ordre 
au moins~2 par rapport à $y$ sont nulles en $p$ jusqu'à l'ordre souhaité.
\end{proof}

\begin{theo}\label{bornes:nodal3}
La décomposition nodale de $\Sigma$ vérifie les propriétés suivantes :
\begin{enumerate}
\item Tous les domaines nodaux rencontrent le bord de la surface $\Sigma$.
\item Les domaines nodaux et les composantes connexes du graphe nodal 
d'une fonction propre sont incompressibles dans la surface $\Sigma$. 
\item Si $\Gamma$ est l'union d'une ou plusieurs composantes connexes
du graphe nodal, alors on a la relation $\chi(\Gamma)\geq\chi(\Sigma)$.
\end{enumerate}
\end{theo}
\begin{proof}
Les premières propriétés énoncées dans ce théorème sont déjà connues
(voir par exemple~\cite{ja14}), mais nous allons rappeler leur démonstration.

 Si un domaine nodal ne rencontre pas le bord de $\Sigma$, alors
la fonction propre est nulle sur toute la frontière du domaine, donc
elle est nulle à l'intérieur du domaine puisqu'elle est harmonique
et par conséquent elle est nulle partout selon la propriété d'unique 
prolongement.

Supposons qu'un domaine nodal $D$ contienne une courbe contractile
dans $\Sigma$ mais pas dans $D$. Alors le disque bordé par cette courbe
contient un domaine nodal distinct de $D$ et qui ne rencontre pas le 
bord de $\Sigma$. Or on vient de voir que c'est impossible, donc
toute courbe de $D$ contractile dans $\Sigma$ est contractile dans $D$.
Le même argument montre que le graphe nodal est incompressible.

Enfin, si on note $\Gamma$ la réunion d'une ou plusieurs composantes connexes
du graphe nodal, et $D_i$ les composantes connexes de $\Sigma\backslash
\Gamma$, la formule d'Euler-Poincaré nous dit que 
\begin{equation}
\chi(\Sigma)=\chi(\Gamma)+\sum_i\chi(D_i).
\end{equation}
Pour un $i$ donné, la caractéristique d'Euler de l'intérieur de $D_i$ est 
au plus égale à~1. Mais comme $D_i$ contient nécessairement un domaine 
nodal de la fonction propre, il rencontre le bord. Donc $D_i$ est
la réunion de son intérieur et d'un (ou plusieurs) intervalle ouvert 
du bord de $\Sigma$, donc $\chi(D_i)\leq0$. Par conséquent,
$\chi(\Gamma)\geq\chi(\Sigma)$.
\end{proof}
\begin{rema}
L'argument d'incompressibilité de l'ensemble nodal permet en fait de
montrer un résultat plus fort mais dont nous ne ferons pas usage : 
l'ensemble nodal ne contient pas de courbe fermée qui borde un domaine.
\end{rema}

\subsection{Majoration de multiplicité}
Le but de cette section est de démontrer les bornes sur la multiplicité
données par les théorèmes~\ref{intro:bornes} et~\ref{intro:sev}.

\begin{proof}[Démonstration du théorème~\ref{intro:bornes}]

La démonstration de l'inégalité~(\ref{intro:borne1}) reprend la technique 
de N.~Nadirashvili \cite{na88} avec les améliorations que permet le 
problème de Steklov. On fixe un entier $k\geq1$ et on note $m$ la
multiplicité de la $k$-ième valeur propre. On se donne aussi un point 
$p\in\partial\Sigma$. Selon le lemme~\ref{bornes:nodal2}, on peut
trouver une fonction non nulle dans le $k$-ième espace propre qui s'annule
à l'ordre~$m-1$ en~$p$. Le point~$p$ est donc l'extrémité de $m-1$ lignes 
nodales (théorème~\ref{bornes:nodal1}).

Parmi les lignes nodales d'extrémité~$p$, il y en a au 
plus $1-\chi(\Sigma)$ dont les deux extrémités sont~$p$. En effet, si on note
$q$ le nombre de lignes allant de $p$ à $p$, la composante connexe de~$p$ 
dans le graphe nodal aura une caractéristique d'Euler au plus égal à $1-q$,
et donc $\chi(\Sigma)\leq 1-q$ d'après le théorème~\ref{bornes:nodal3}.
On déduit de cette remarque que~$p$ est l'extrémité d'au moins 
$m+\chi(\Sigma)-2$ lignes distinctes.

On peut alors majorer $m$ en appliquant la formule d'Euler-Poincaré 
à la surface $\Sigma$. Contrairement à ce qui est fait dans la
démonstration du théorème~\ref{bornes:nodal3}, on considérera la 
surface \emph{ouverte}, c'est-à-dire en ignorant le bord et les sommets
du graphe nodal qui s'y trouve. On notera $\dot\Gamma$ le graphe
ainsi obtenu. On a alors

\begin{equation}
\chi(\Sigma)=\chi(\dot\Gamma)+\sum_i\chi(D_i).
\end{equation}
Chaque domaine nodal vérifie $\chi(D_i)\leq1$ donc $\sum_i\chi(D_i)\leq k+1$
selon le théorème de Courant. Comme $\dot\Gamma$ est la réunion d'au moins
$m+\chi(\Sigma)-2$ lignes nodales homéomorphes à des intervalles ouverts,
sa caractéristique d'Euler vérifie $\chi(\dot\Gamma)\leq-m-\chi(\Sigma)+2$.
On obtient finalement que
\begin{equation}
m\leq k+3-2\chi(\Sigma).
\end{equation}

\end{proof}

Avant d'entamer la démonstration du théorème~\ref{intro:sev}, on va
montrer une propriété des fonctions propres qui sera utile pour établir 
l'inégalité~(\ref{intro:borne3}):
\begin{lemme}\label{bornes:lemsev}
 Soit $f$ une fonction propre de Steklov et $c>0$ un réel positif. Chaque 
composante connexe de l'ensemble $f^{-1}(]-c,+\infty[)$ contient au moins un
domaine nodal positif de $f$. En particulier, si $f$ est associée à la 
première valeur propre non nulle, cet ensemble est connexe.
\end{lemme}
\begin{proof}
Soit $D$ une composante connexe de $f^{-1}(]-c,+\infty[)$ qui ne contient pas
de domaine nodal positif.  Comme $f$ est harmonique, le maximum de $f$ 
sur $D$ est atteint en un point $x$ situé sur le bord de la surface.

On a nécessairement $f(x)<0$. En effet, si $f(x)>0$ alors $D$ contient
un domaine nodal positif. Si $f(x)=0$, alors $x$ est l'extrémité d'une
ligne nodale de $f$ qui traverse le domaine, donc $f$ change de signe
dans $D$, en particulier elle prend des valeurs strictement positive dans
$D$ ce qui est impossible.

Au point $x$, la fonction $f$ vérifie l'équation aux valeurs propres
$\partial f/\partial\nu=\rho\sigma f$. Par conséquent, 
$(\partial f/\partial\nu)(x)$ est strictement négatif. Comme $\nu$ est 
un vecteur \emph{sortant} normal au bord, on a une contradiction avec 
le fait que $x$ soit le maximum de $f$ sur $D$.

Pour finir, d'après ce qui précède le nombre de composantes connexes
de $f^{-1}(]-c,+\infty[)$ est majoré par le nombre de domaines nodaux 
positifs de $f$. Or, si $f$ est associée à la première valeur propre 
non nulle elle a exactement deux domaines nodaux, un positif et un négatif.
Par conséquent, $f^{-1}(]-c,+\infty[)$ n'a qu'une composante connexe.
\end{proof}

Le théorème~\ref{intro:sev} repose sur un théorème de B.~Sévennec que nous
allons rappeler ici:
\begin{theo}[\cite{se02}, théorème~5]\label{bornes:sev}
Si $\Sigma$ est une surface close de caractéristique d'Euler strictement 
négative et $E$ un espace de fonction continues sur $\Sigma$ tel que
pour toute fonction $f\in E\backslash\{0\}$, les ensembles
$f^{-1}]0,+\infty[$ et $f^{-1}[0,+\infty[$ sont connexes et non vides,
alors $\textrm{Dim}(E)\leq 5-\chi(\Sigma)$.
\end{theo}
\begin{rema}\label{bornes:max}
Le principe du maximum assure que pour une fonction propre de Steklov $f\neq0$,
l'ensemble $f^{-1}[0,+\infty[$ n'a pas de composante connexe disjointe
de $f^{-1}]0,+\infty[$. 
\end{rema}
\begin{rema}\label{bornes:reg}
Pour appliquer le théorème~\ref{bornes:sev} aux fonctions propres de 
$\sigma_1$, on utilise le théorème de Courant qui assure que ces fonctions
propres ont exactement deux domaines nodaux. Or, G.~Alessandrini a montré
dans \cite{al98} qu'il reste valable en dimension~2 si $g$ et $\gamma$
sont seulement $L^\infty$. Comme on n'a pas besoin de la régularité des
fonctions propres sur le bord de la surface, on peut aussi supposer 
que $\rho$ est $L^\infty$.
\end{rema}
\begin{proof}[Démonstration du théorème~\ref{intro:sev}]
On commence par l'inégalité~(\ref{intro:borne2}). En collant un disque sur
chaque composante du bord de $\Sigma$, on obtient une
surface close $\overline\Sigma$ de caractéristique d'Euler $\bar\chi=\chi+l$.

On prolonge chaque fonction propre $f$ de la première
valeur propre sur chacun de ces disques de la manière suivante :
en notant $p$ le centre du disque et $x$ un point générique sur le bord, 
on fixe $f(p)=0$ et on interpole linéairement $f$ sur le segment $[p,x]$.
La fonction $f$ ainsi prolongée est continue, et elle a deux domaines nodaux,
un positif et un négatif
(en effet, comme elle est de signe constant sur $]px]$, les points
intérieurs aux disques appartiennent aux mêmes domaines nodaux que les 
points de $\Sigma$). Compte tenu de la remarque~\ref{bornes:max}, cette
construction assure que les ensembles $f^{-1}[0,+\infty[$ et 
$f^{-1}]-\infty,0]$ sont eux aussi connexes. D'après le
théorème de Sévennec, cet espace est de dimension au plus
$5-\bar\chi$. On a donc $m_1\leq5-\bar\chi=5-\chi-l$.

Pour montrer l'inégalité~(\ref{intro:borne3}), on va encore appliquer
le théorème~5 de Sévennec mais à un espace plus grand : on considère 
l'espace $E$ engendré par les fonctions propres prolongées à $\overline\Sigma$
comme précédemment et par la fonction $\varphi$ définie par $\varphi\equiv 1$
sur $\Sigma$, $\varphi(p)=-1$ et en interpolant la fonction de manière
affine sur les rayons du disque (comme on suppose que $l=1$, il n'y a 
qu'un disque et le point $p$ est unique). L'espace $E$ est constitué de 
fonction continues, on doit donc montrer que ces fonctions (à l'exception 
de la fonction nulle) ont exactement deux 
domaines nodaux, un positif et un négatif.

Une fonction de $E$ est une combinaison linéraire $f_{a,b}=
a\cdot f+b\cdot\varphi$,
$a,b\in\R$, $f$ étant une fonction propre prolongée à $\overline\Sigma$.
Si $a$ ou $b$ est nul, il est clair que la fonction a un seul domaine 
nodal positif. 

Sans nuire à la généralité, on peut supposer que $a=1$ et que $b$ est 
strictement positif.  La fonction $f_{a,b}$ est égale à $f+b$ sur
$\Sigma$ et $f_{a,b}(p)=-b$. Selon le lemme~\ref{bornes:lemsev} appliqué
à la fonction $f$, l'ensemble $f_{a,b}^{-1}(]0,+\infty[)=
f^{-1}]-b,+\infty[$ restreint à 
$\Sigma$ est connexe. Comme $f_{a,b}$ est affine sur les rayons du disque et
négative en son centre, les points du disque où $f_{a,b}$ est positive
sont connectés au domaine nodal positif sur $\Sigma$. La fonction
$f_{a,b}$ n'a donc globalement qu'un seul domaine nodal positif.

En restriction à $\Sigma$, la fonction $f_{a,b}$ peut avoir plusieurs
domaines nodaux négatifs. Cependant, ils sont connectés à $p$ par des
rayons sur lesquels $f_{a,b}$ est négative. Il n'y a donc qu'un seul
domaine nodal négatif sur $\overline\Sigma$. Comme $E$ est de 
dimension $m_1+1$, l'application du théorème de Sévennec donne 
$m_1+1\leq5-\bar\chi=4-\chi$, soit $m_1\leq3-\chi$.
\end{proof}

\subsection{Première valeur propre en petit genre}
On commencera dans ce paragraphe par démontrer le théorème~\ref{intro:thl},
puis on montrera séparément les calculs des trois valeurs de $m_k$
annoncées dans par le théorème~\ref{intro:thp}, les deux premières
dans ce paragraphe et la 3\ieme{} dans le suivant.

\begin{proof}[Démonstration du théorème~\ref{intro:thl}]
Pour majorer la multiplicité il s'agit, comme dans \cite{fs12}, 
\cite{ja14} et \cite{kkp12}, de se ramener à la démonstration utilisée
dans le cas des opérateurs de Schrödinger sur les surfaces closes. 
La minoration découle du théorème~\ref{intro:thpresc}.

Dans le cas de $\mathbb T^2_p$, on fixe un point $x$ intérieur à la surface
et on suppose que la multiplicité de $\sigma_1$ est au moins~7. Il existe
alors un espace de dimension~2 de fonctions propres qui
s'annulent à l'ordre~3 en $x$, et on note $f_\theta$ le cercle unité de
cet espace. Si on ferme la surface en contractant 
chaque bord de $\mathbb T^2_p$ en un point, les lignes nodales de
$f_\theta$ qui atteignent le bord se prolongent sur le quotient en
des courbes $C^1$ par morceaux. On peut alors reprendre les arguments
de \cite{be80} pour montrer que sur le quotient, l'ensemble nodal est 
formé d'exactement trois lacets disjoints non homotopes et aboutir
à une contradiction en faisant varier $\theta$.

Sur $\mathbb K^2_p$, on se ramène de la même manière à la démonstration
du théorème~4.1 de \cite{cdv87} en quotientant chaque bord sur un point.
\end{proof}

Avant de traiter le cas de $\mathbb P^2_2$, on va montrer que 
$m_1(\mathbb M^2)=4$ par une méthode
légèrement différente de celle du théorème~\ref{intro:bornes}. Cette
démonstration servira de base pour les deux propositions qui vont suivre.

 On commence
par montrer qu'il existe une fonction propre pour laquelle il existe au 
moins~6 extrémités de lignes nodales qui rejoignent le bord.
Soit $x$ un point du bord de $\mathbb M^2$. Si la multiplicité de $\sigma_1$
est supérieure ou égale à~5, il existe une fonction propre pour laquelle
$x$ est l'extrémité d'au moins quatre lignes nodales. S'il est l'extrémité
d'au moins~5 lignes, alors il y a au moins six extrémités de lignes qui
rejoignent le bord puisque ce nombre est toujours pair. S'il est 
l'extrémité d'exactement quatre lignes alors le signe de la fonction propre
est constant le long du bord au voisinage de $x$, par conséquent 
la fonction propre change de signe ailleurs sur le bord (sinon la fonction
est positive ou nulle partout sur le bord, donc partout sur $\mathbb M^2$) 
et il y a donc
au moins six extrémités de lignes le long du bord. Si on définit une
application $\mathbb M^2\mapsto\mathbb P^2$ par contraction du bord
de $\mathbb M^2$ sur un point (qu'on notera encore $x$), on peut considérer
l'image de la décomposition nodale de $\mathbb M^2$ par cette application :
elle décompose $\mathbb P^2$ en deux domaines, et $x$ est l'extrémité
de six arcs nodaux. Or G.~Besson a montré dans \cite{be80} qu'une telle
décomposition de $\mathbb P^2$ est impossible.

\begin{prop}
$m_1(\mathbb P^2_2)=4$.
\end{prop}
\begin{proof}
On représente $\mathbb P^2_2$ comme étant $\mathbb M^2$ privé d'un disque.
On parlera de bord de $\mathbb M^2$ et de bord du disque pour désigner
les deux bords de $\mathbb P^2_2$.

Le point de départ de la démonstration est le même que celle qui précède. 
On suppose que la multiplicité de $\sigma_1(\mathbb P^2_2)$ est~5. Étant
donné un point $x$ du bord de $\mathbb M^2$, il existe une fonction
propre $f$ telle que $x$ soit localement l'extrémité de quatre lignes nodales.
On a vu précédemment qu'il ne peut pas y avoir d'autres extrémités 
de lignes nodales sur le bord de $\mathbb M^2$. Le signe de $f$ le long
de ce bord est donc constant et on le supposera positif. Comme les deux
domaines nodaux de $f$ rencontrent le bord de $\mathbb P^2_2$, le domaine
négatif rencontre nécessairement le bord du disque.

Si on considère un arc qui part orthogonalement de $x$ (avec par conséquent
deux lignes nodales de part et d'autre, selon le théorème~\ref{bornes:nodal1})
et qui rejoint un autre point du bord en restant dans le domaine nodal
positif, il découpe nécessairement le ruban $\mathbb M^2$ en un rectangle
comme sur la figure~\ref{bornes:p22}. La décomposition nodale de 
$\mathbb P^2_2$ est donc nécessairement topologiquement équivalente
à celle de cette figure, à ceci près que le bord du disque peut être
entièrement contenu dans le domaine nodal négatif.

\begin{figure}[h]
\begin{center}
\begin{picture}(0,0)%
\includegraphics{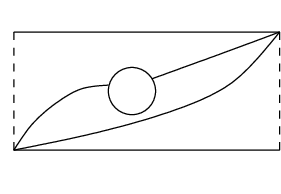}%
\end{picture}%
\setlength{\unitlength}{4144sp}%
\begingroup\makeatletter\ifx\SetFigFont\undefined%
\gdef\SetFigFont#1#2#3#4#5{%
  \reset@font\fontsize{#1}{#2pt}%
  \fontfamily{#3}\fontseries{#4}\fontshape{#5}%
  \selectfont}%
\fi\endgroup%
\begin{picture}(2142,1471)(346,-845)
\put(361,-781){\makebox(0,0)[lb]{\smash{{\SetFigFont{12}{14.4}{\rmdefault}{\mddefault}{\updefault}{\color[rgb]{0,0,0}$x$}%
}}}}
\put(2386,479){\makebox(0,0)[lb]{\smash{{\SetFigFont{12}{14.4}{\rmdefault}{\mddefault}{\updefault}{\color[rgb]{0,0,0}$x$}%
}}}}
\put(2071,-331){\makebox(0,0)[lb]{\smash{{\SetFigFont{12}{14.4}{\rmdefault}{\mddefault}{\updefault}{\color[rgb]{0,0,0}$+$}%
}}}}
\put(676,119){\makebox(0,0)[lb]{\smash{{\SetFigFont{12}{14.4}{\rmdefault}{\mddefault}{\updefault}{\color[rgb]{0,0,0}$+$}%
}}}}
\put(1714,-16){\makebox(0,0)[lb]{\smash{{\SetFigFont{12}{14.4}{\rmdefault}{\mddefault}{\updefault}{\color[rgb]{0,0,0}$-$}%
}}}}
\end{picture}%
\end{center}
\caption{Décomposition nodale de $\mathbb P^2_2$%
\label{bornes:p22}}
\end{figure}
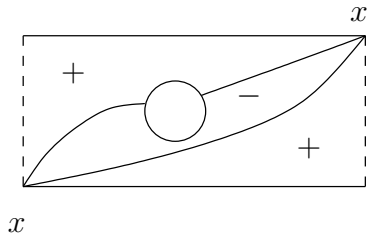

Après avoir contracté le bord du disque sur un point qu'on notera $p$, 
on pourra donc toujours trouver dans $\mathbb M^2$ un lacet $\gamma_x$
d'extrémité $x$, unique à homotopie près, contenu dans le domaine nodal 
négatif, passant par $p$ et homotope à un générateur du groupe fondamental 
de $\mathbb M^2$.

 Supposons maintenant qu'on déplace continûment le point $x$. Le lacet
$\gamma_x$ peut alors se déformer continûment avec la contrainte de toujours
passer par $p$. Si on fait faire à $x$ un tour complet du bord de 
$\mathbb M^2$ en partant d'un point $x_0$, on revient à la même décomposition 
de $\mathbb P^2_2$, donc le lacet $\gamma_x$ devient homotope au lacet initial
$\gamma_{x_0}$. Or, chacun des arcs de $\gamma_{x_0}$ allant de $x_0$ à $p$
serait homotope à sa concaténation avec un générateur du groupe fondamental 
du bord, ce qui est impossible.
\end{proof}

\begin{prop}
$m_1(\mathbb T^2_1)=5$
\end{prop}
\begin{proof}
On procède par l'absurde en supposant que $\sigma_1(\mathbb T^2_1)$
est de multiplicité~6 pour une métrique $g$ et des densités  $\rho, \gamma$ 
données.

Soit $x$ un point du bord de $\mathbb T^2_1$. Selon le 
théorème~\ref{bornes:nodal1} et le lemme~\ref{bornes:nodal2}, il existe 
une fonction  propre $f$ qui s'annule à l'ordre~5 en $x$ et telle que
cinq arcs nodaux partent de $x$. Notons que cette fonction est unique :
dans le cas contraire, on pourrait en choisir une s'annulant à l'ordre~6
et on aurait une contradiction comme dans le théorème~\ref{intro:bornes}.
Comme les extrémités de lignes nodales rejoignant le bord sont nécessairement
en nombre pair, il existe une 6\ieme{} extrémité en un point $x'$ du bord,
$x'$ étant distinct de $x$ (sinon $f$ s'annulerait à l'ordre~6 en $x$).

Comme dans la démonstration précédente, on va déplacer le point $x$ le long
du bord. Les arguments développés par G.~Besson dans \cite{be80} permettront
d'aboutir à une contradiction. Si on contracte le bord sur un point 
---~notons-le $\bar x$~--- alors il y a six arcs nodaux partant de $\bar x$.
On sait alors que l'ensemble nodal est la réunion de trois lacets
non homotopes entre eux et qui ne s'intersectent qu'en $\bar x$ (cf.
la démonstration du théorème~3.C.1 dans \cite{be80}). Cette remarque
permet alors de transposer le reste de la démonstration de \cite{be80} :
si $x$ se déplace continûment en partant d'un point $x_0$, on peut
construire une homotopie entre les 
lignes nodales de la fonction propre $f_x$ correspondante et les lignes 
nodales de $f_{x_0}$ (il est crucial ici que $x\neq x'$ quel que
soit $x$). Après un tour complet de $x$ le long du bord, les classes 
d'homotopies des lignes nodales sont donc les conjuguées des lignes de
$f_{x_0}$ par la classe d'homotopie du bord. Or elles doivent être identiques
aux lignes nodales de $f_{x_0}$. Il y a contradiction car les classes
d'homotopies des lignes nodales sont des classes d'homotopies
non triviales dans le tore $\mathbb T^2$ obtenu par contraction du bord,
elles ne commutent donc pas avec la classe du bord.
\end{proof}

\subsection{Deuxième valeur propre du disque}
Pour finir, on montre la majoration de $m_2(\mathbb D^2)$ annoncée par
le théorème~\ref{intro:thdisque}.
On sait déjà, d'après le théorème~\ref{intro:bornes}, que 
$m_2(\mathbb D^2)\leq3$. On va supposer qu'il y a égalité pour aboutir à
une contradiction.

On se donne une métrique $g$ et une densité $\gamma$ telle que la 
multiplicité de $\sigma_2(\mathbb D^2)$ soit~3. Selon le théorème de 
Courant, une fonction propre de $\sigma_2(\mathbb D^2)$ a deux ou 
trois domaines nodaux. Comme ces domaines sont simplement connexes, la
décomposition nodale du disque est nécessairement topologiquement
équivalente à l'une des trois indiquées sur la figure~\ref{bornes:disque}.
\begin{figure}[h]
\begin{center}
\begin{picture}(0,0)%
\includegraphics{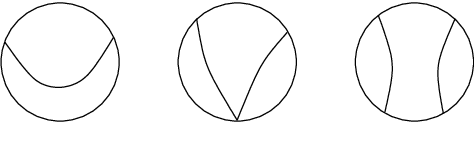}%
\end{picture}%
\setlength{\unitlength}{4144sp}%
\begingroup\makeatletter\ifx\SetFigFont\undefined%
\gdef\SetFigFont#1#2#3#4#5{%
  \reset@font\fontsize{#1}{#2pt}%
  \fontfamily{#3}\fontseries{#4}\fontshape{#5}%
  \selectfont}%
\fi\endgroup%
\begin{picture}(3616,1241)(668,-1070)
\put(991,-1006){\makebox(0,0)[lb]{\smash{{\SetFigFont{12}{14.4}{\rmdefault}{\mddefault}{\updefault}{\color[rgb]{0,0,0}(a)}%
}}}}
\put(2341,-1006){\makebox(0,0)[lb]{\smash{{\SetFigFont{12}{14.4}{\rmdefault}{\mddefault}{\updefault}{\color[rgb]{0,0,0}(b)}%
}}}}
\put(3691,-1006){\makebox(0,0)[lb]{\smash{{\SetFigFont{12}{14.4}{\rmdefault}{\mddefault}{\updefault}{\color[rgb]{0,0,0}(c)}%
}}}}
\end{picture}%
\end{center}
\caption{Domaines nodaux sur le disque%
\label{bornes:disque}}
\end{figure}
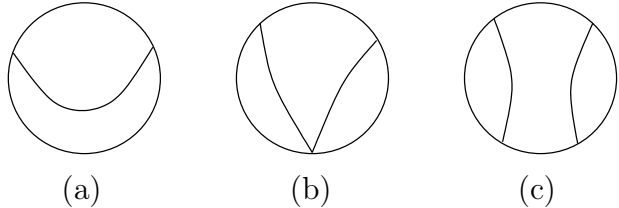

On considère la sphère unité de l'espace propre de $\sigma_2(\mathbb D^2)$,
qu'on notera~$S^2$, et on note respectivement $D_a$ (resp. $D_b$ et $D_c$)
l'ensemble des fonctions propres de~$S^2$ dont la décomposition nodale
est celle de la figure~\ref{bornes:disque}.a (resp.
\ref{bornes:disque}.b et \ref{bornes:disque}.c). La contradiction
découlera de l'étude de la partition de $S^2$ ainsi formée. 

Commençons par noter que, comme la multiplicité est égale à~3, le
théorème~\ref{bornes:nodal1} et le lemme~\ref{bornes:nodal2} montrent
qu'il existe
nécessairement des fonctions propres réalisant la situation~(b) de la
figure~\ref{bornes:disque} et donc $D_b$ est non vide. Plus précisément,
pour chaque point $x$ de $\partial\mathbb D^2$, il existe une droite de 
fonction propre s'annulant à l'ordre~2 en $x$, et cette droite varie 
continûment avec $x$ (elle est unique car dans le cas contraire, on pourrait
trouver une fonction s'annulant à l'ordre~3 en $x$). Comme cette droite coupe
$S^2$ en deux points, l'ensemble $D_b$ est la réunion de deux cercles,
l'un correspondant aux fonctions ayant un seul domaine nodal positif, l'autre
à celles en ayant deux (en particulier, ces cercles sont disjoints).

 Le complémentaire de $D_b$ est donc formé de trois composantes connexes,
deux qui sont antipodales et homéomorphes à des disques et une homéomorphe
à un cylindre. Comme $D_c$ a nécessairement deux composantes connexes
antipodales, l'une formée des fonctions ayant un domaine nodal positif
et l'autre des fonctions en ayant deux, $D_c$ est la réunion des deux 
disques et $D_a$ est le cylindre.

On choisit l'une des deux composantes de $D_c$, par exemple celle
dont les fonctions ont deux domaines nodaux positifs, qu'on notera $D_c^+$,
et on construit une application de $D_c^+$ dans le cercle de la manière 
suivante : on se donne une orientation sur $\partial\mathbb D^2$, et 
pour chaque fonction de $D_c^+$ on considère la paire de points
de $\partial\mathbb D^2$ où la fonction s'annule en décroissant (par 
rapport à l'orientation du bord). Ces deux points sont nécessairement
disjoints (sinon il y aurait au moins trois extrémités de lignes
nodales se rejoignant en un même points du bord, ce qu'on a déjà exclu).
On obtient ainsi une application continue partant de $D_c^+$ et dont
l'image est $S^1\times S^1$ privé de la diagonale et quotientée par
$(x,y)\sim(y,x)$. Cette image est homotope à un cercle, ce qui 
permet de définir une application $D_c^+\to S^1$. Or, le long du
bord de $D_c^+$, c'est-à-dire d'une des composantes de $D_b$, cette 
application est homotopiquement non triviale, ce qui fournit la contradiction
cherchée.

\providecommand{\bysame}{\leavevmode ---\ }
\providecommand{\og}{``}
\providecommand{\fg}{''}
\providecommand{\smfandname}{\&}
\providecommand{\smfedsname}{\'eds.}
\providecommand{\smfedname}{\'ed.}
\providecommand{\smfmastersthesisname}{M\'emoire}
\providecommand{\smfphdthesisname}{Th\`ese}

\end{document}